\documentclass[11pt,a4paper]{amsart}
\usepackage[T1]{fontenc}
\usepackage{lmodern}
\usepackage{amsmath,amssymb,amsthm,mathtools}
\usepackage{mathrsfs}
\usepackage{booktabs}
\newcommand{\rev}[1]{{#1}}
\newenvironment{revision}{\par\begingroup}{\par\endgroup}
\usepackage[margin=30mm]{geometry}
\usepackage{microtype}
\usepackage[hidelinks]{hyperref}
\hypersetup{pdftitle={Linear hereditary quasirandomness},pdfauthor={Jun Yue}}
\allowdisplaybreaks[1]
\numberwithin{equation}{section}

\newtheorem{theorem}{Theorem}[section]
\newtheorem{proposition}[theorem]{Proposition}
\newtheorem{lemma}[theorem]{Lemma}
\newtheorem{corollary}[theorem]{Corollary}

\newtheorem*{linearconjecture}{Linear dependence conjecture}
\theoremstyle{remark}
\newtheorem{remark}[theorem]{Remark}

\newcommand{\R}{\mathbb R}
\newcommand{\one}{\mathbf 1}
\newcommand{\cut}[1]{\left\lVert #1\right\rVert_{\square}}
\newcommand{\norm}[2]{\left\lVert #1\right\rVert_{#2}}
\newcommand{\tr}{\operatorname{tr}}
\newcommand{\Def}{\operatorname{def}}
\newcommand{\Disc}{\operatorname{disc}}
\newcommand{\Pstar}{\mathcal P^{*}}

\title[Linear hereditary quasirandomness]{Linear hereditary quasirandomness}
\author{Jun Yue}
\address{School of Mathematical Sciences, Tiangong University, Tianjin, China}
\email{yuejun06@126.COM}
\date{}

\begin{document}

\begin{abstract}
Conlon, Fox and Sudakov (\emph{Math. Proc. Cambridge Philos. Soc.}
\textbf{164} (2018), 385--399) conjectured that hereditary subgraph
counts force quasirandomness with linear dependence on the counting
error. We prove this for every graph containing an edge.
Let \(H\) have \(r\) vertices and \(m\ge1\) edges, and fix
\(p\in(0,1)\). If every subset \(S\) of an \(n\)-vertex graph \(G\)
contains \(p^m|S|^r\pm\delta n^r\) labelled injective copies of \(H\),
then \(\lvert e_G(A,B)-p|A||B|\rvert\le C_{H,p}\delta n^2\)
for all \(A,B\subseteq V(G)\), where \(e_G(A,B)\) counts ordered
adjacent pairs. The proof extracts averaged rooted counts and
represents each nonlinear term by a linear map in one selected
edge factor. Composing two such maps gives an edge substitution
with a superlinear cut-norm bound. A separate star expansion
gives stronger local estimates for triangle-free graphs.
The linear order is best possible even when
\(2e(G)/n^2=p+O(n^{-2})\).
\end{abstract}

\maketitle

\section{Introduction}

All graphs considered in this paper are finite, simple, and undirected. Fix a graph \(H\)
with \(r\) vertices and \(m\ge1\) edges, and let \(p\in(0,1)\).
For \(S\subseteq V(G)\), let \(N_H(G[S])\) denote the number of
injective edge-preserving maps from \(V(H)\) to \(S\).
Thus copies are labelled and are not required to be induced.
We say that an \(n\)-vertex graph \(G\) satisfies
\(\Pstar_{H,p}(\delta)\) if
\begin{equation}\label{eq:hereditary}
 \bigl|N_H(G[S])-p^m|S|^r\bigr|\le\delta n^r
 \qquad\text{for every }S\subseteq V(G).
\end{equation}

To compare this hereditary counting condition with uniform
edge distribution, for \(A,B\subseteq V(G)\) we write
\(e_G(A,B)=|\{(x,y)\in A\times B:\{x,y\}\in E(G)\}|\).
This counts ordered adjacent pairs, so that
\(e_G(S,S)=2e(G[S])\). The normalized edge discrepancy is
\[
 \Disc_p(G)=\frac1{n^2}\max_{A,B\subseteq V(G)}
 \bigl|e_G(A,B)-p|A||B|\bigr|.
\]
Throughout, \(H\) and \(p\) are fixed. Constants may depend on
\(H\) and \(p\), as indicated by their subscripts, but are
independent of \(G\), \(n\), and \(\delta\).

The connection between subgraph counts and uniform edge
distribution is central to the theory of quasirandom graphs.
Chung, Graham and Wilson~\cite{CGW} established the equivalence
of several counting, spectral, and edge-distribution conditions
for dense quasirandomness. Simonovits and S\'os~\cite{SS}
subsequently showed that the hereditary counting condition
\eqref{eq:hereditary} characterizes quasirandomness for every
graph \(H\) containing an edge. Their proof uses the regularity
lemma and yields tower-type quantitative bounds. Conlon, Fox
and Sudakov~\cite{CFS} improved this dependence to polynomial
for general \(H\) and to linear when \(H\) is complete.
In their concluding discussion, they asked whether linear
dependence holds for every \(H\), highlighting the case
\(H=C_4\).

\begin{linearconjecture}[Conlon--Fox--Sudakov {\cite[Section~4]{CFS}}]
For every graph \(H\) with at least one edge and every \(p\in(0,1)\),
there is \(C=C(H,p)>0\) such that, for each \(\delta>0\), every
sufficiently large graph satisfying \(\Pstar_{H,p}(\delta)\)
has \(\Disc_p(G)\le C\delta\).
\end{linearconjecture}

The question above is the unnumbered hereditary-counting question
posed in Section~4 of~\cite{CFS}. It is distinct from their
Conjecture~4.1, which concerns a global forcing condition;
related global forcing conditions are also studied by Reiher
and Schacht~\cite{RS}. Our main theorem gives an affirmative
answer to the hereditary-counting question.

\begin{theorem}\label{thm:main}
Let \(H\) be a fixed graph with at least one edge, and let
\(p\in(0,1)\). There exists a constant \(C=C(H,p)>0\) such that,
for every \(n\ge1\), every \(\delta\ge0\), and every \(n\)-vertex
graph \(G\),
\[
 \Pstar_{H,p}(\delta)
 \quad\Longrightarrow\quad
 \Disc_p(G)\le C\delta.
\]
The graph \(H\) may be disconnected and may contain isolated
vertices. The same conclusion holds if \(N_H(G[S])\) in
\eqref{eq:hereditary} counts all edge-preserving maps
\(V(H)\to S\), without requiring injectivity.
\end{theorem}

Theorem~\ref{thm:main} establishes linear dependence on \(\delta\)
for each fixed pair \(H,p\). The local estimates below have
constants bounded by \(c_Hp^{-6m}\), but the passage to the global
result also uses a fixed-accuracy threshold from the hereditary
quasirandomness theorem. Consequently, the proof does not give
a polynomial bound in \(p^{-1}\) for \(C(H,p)\). The conclusion holds for every finite \(n\), with no additive \(n^{-1}\) error. This is achieved by using a discrepancy lower bound for simple graphs to absorb the finite-size errors.

Taking \(A=B=S\) in the definition of \(\Disc_p(G)\) gives
the hereditary edge-count condition. We therefore obtain the
following special case, which includes \(H=C_4\), the case
highlighted in~\cite{CFS}.

\begin{corollary}\label{cor:linear}
Let \(H\) be a fixed triangle-free graph with at least one edge,
and let \(p\in(0,1)\). There exists a constant \(C=C(H,p)>0\)
such that, for every \(n\ge1\), every \(\delta\ge0\), and every
\(n\)-vertex graph \(G\),
\[
 \Pstar_{H,p}(\delta)
 \quad\Longrightarrow\quad
 \Pstar_{K_2,p}(C\delta).
\]
\end{corollary}

A stronger counting hypothesis allows a separate vertex set
for each vertex of the counted graph. For a graph \(F\) with
vertex set \([s]\), let \(N_F(G;S_1,\ldots,S_s)\) denote the
number of injective edge-preserving maps \(\varphi\colon V(F)\to V(G)\)
such that \(\varphi(i)\in S_i\) for every \(i\in[s]\). He~\rev{\cite[Theorem~5 and Lemma~8]{He}} proved a linear quasirandomness bound from prescribed-set counts for every fixed graph \(F\) with at least one edge, using homomorphism
counts rather than injective counts. This formulation is part
of the configuration-counting approach of Reiher and Schacht
discussed in~\cite[Section~2]{He}. The essential distinction between the two hypotheses is the freedom to choose the sets \(S_i\) independently; both count
labelled copies. Polarizing \eqref{eq:hereditary} controls
averages over assignments of weights to vertices, rather than
each prescribed assignment separately. Nevertheless,
Theorem~\ref{thm:main}, together with the standard counting
lemma, yields the following prescribed-set estimate.

\begin{corollary}\label{cor:prescribed}
Let \(H\) be a fixed graph with at least one edge, let \(F\)
be a fixed graph with vertex set \([s]\) and \(\ell\) edges,
and let \(p\in(0,1)\). There exists a constant
\(C=C(H,F,p)>0\) such that, for every \(n\ge1\) and every
\(\delta\ge0\), any \(n\)-vertex graph \(G\) satisfying
\(\Pstar_{H,p}(\delta)\) also satisfies
\[
 \left|N_F(G;S_1,\ldots,S_s)
       -p^\ell\prod_{i=1}^s|S_i|\right|
 \le C\delta n^s
\]
for all \(S_1,\ldots,S_s\subseteq V(G)\), with no disjointness
assumption. The conclusion remains valid if the hypothesis,
the conclusion, or both are formulated using homomorphism
counts.
\end{corollary}

Taking \(F=H\) gives the conversion proposed
in~\cite[Conjecture~7]{He} for \(0<p<1\).
Corollary~\ref{cor:prescribed} follows from
Theorem~\ref{thm:main} by the standard counting lemma,
without a separate argument to recover edge discrepancy.

\begin{revision}
Related recent work treats different counting hypotheses.
Li and Liu~\cite{LL} study individual edge-rooted densities,
with a stability estimate for kernels bounded away from zero.
Wang, Wang and Yue~\cite{WWY} classify quasirandomness conditions
based on balanced and repeated-class counts. These results
concern different rooted or partition constraints; the question
here is linear error dependence from counts in every single
vertex subset.

The standard ingredients include polarization, cut-norm
counting estimates, and qualitative hereditary quasirandomness.
The main quantitative argument below is the composition of
selected-edge corrections and the resulting local recovery
estimate. Corollaries~\ref{cor:linear} and~\ref{cor:prescribed}
are consequences of that estimate, rather than independent
recovery theorems.
\end{revision}

\medskip
\noindent\textbf{Proof overview.}
The key step is to represent each nonlinear rooted term
exactly by selecting one edge factor and fixing all the
remaining factors. This defines a linear map on kernels.
Unlike the derivative of the counting functional, which
sums over all choices of the selected edge, this representation
uses a single chosen edge in each term.

An explicit marginal
inverse then inverts the single-edge contribution, including
its one-variable terms. Composing the resulting correction
map with itself corresponds to replacing an edge of one graph
by a two-root density of another. The substituted graph is
simple and cannot consist of a clique together with isolated
vertices. Its contribution is therefore controlled either
by one of five induced patterns or by a separate estimate
for large stars. This yields the local bound
\(d\le c_Hp^{-6m}(\xi+d^{5/4})\), where
\(d=\cut{W-p}\) and \(\xi=\Def_{H,p}(W)\).
The qualitative hereditary quasirandomness theorem ensures
that \(d\) is sufficiently small to absorb the superlinear
term.

For triangle-free \(H\), a separate star expansion gives the
stronger estimate \(d\le c_Hp^{-3m}(\xi+d^{3/2})\).
If \(H\) also contains no four-cycle, the remainder improves
further to \(d^2\). These estimates make explicit how the
absence of short cycles strengthens the local bound.

Grouping star terms near a constant kernel already appears
in Lov\'asz's proof of the local Simonovits--Sidorenko
theorem~\cite[proof of Theorem~4.1]{Lovasz}, which establishes
local lower bounds for bipartite graph densities.
In the present setting, the error estimates must hold
uniformly for signed vertex weights, and the triangle-free
graph \(H\) need not be bipartite. For general \(H\), the
edge-substitution argument handles the clique terms that
are not controlled by the star expansion.

The remainder of the paper is organized as follows.
Section~\ref{sec:tools} develops the rooted-count and
signed-kernel estimates used throughout the paper.
Section~\ref{sec:general} establishes the general local estimate,
with the composed-correction bound stated separately in
Lemma~\ref{lem:squaredcorrection}.
Section~\ref{sec:reduction} applies this estimate to prove
Theorem~\ref{thm:main} and Corollary~\ref{cor:prescribed}.
Section~\ref{sec:proof} derives stronger local estimates for
triangle-free graphs, independently of the proof of the
general theorem.
Section~\ref{sec:remarks} discusses sharpness and parameter
dependence. Finally, Appendix~\ref{sec:fourcycle} provides
the explicit expansions for four- and five-cycles.

\section{Rooted counts and signed-kernel estimates}\label{sec:tools}

We work on a finite probability space \((\Omega,\mu)\), with
all points of measure zero discarded. Unless otherwise stated,
all norms and integrals are taken with respect to the
appropriate products of \(\mu\). For a kernel \(K\), define
its cut norm by
\[
 \cut{K}=\max_{A,B\subseteq\Omega}
 \left|\int_{A\times B}K(x,y)\,d\mu(x)d\mu(y)\right|,
\]
and let \(T_K\) be the associated integral operator,
\(T_Kf(x)=\int K(x,y)f(y)\,d\mu(y)\).
For every real kernel \(K\), the cut norm and the
\(L^\infty\)-to-\(L^1\) operator norm satisfy
\begin{equation}\label{eq:cutdual}
 \cut{K}\le
 \sup_{\substack{\norm{f}{\infty}\le1\\\norm{g}{\infty}\le1}}
 \left|\int f(x)K(x,y)g(y)\,d\mu(x)d\mu(y)\right|
 =\norm{T_K}{\infty\to1}
 \le4\cut{K}.
\end{equation}
\rev{This is the standard cut-norm comparison; see
\cite[Lemma~2.1]{AN}. On a finite probability space it follows
from the matrix version by incorporating the atom weights
into the matrix entries. The middle equality is
\(L^1\)--\(L^\infty\) duality.}

This framework includes finite graphs: if \(W\) is the
adjacency kernel of \(G\) on the uniform probability space
\(V(G)\), with \(W(x,x)=0\), then
\(\cut{W-p}=\Disc_p(G)\).
Thus edge discrepancy is precisely the cut norm of the
deviation from the constant kernel \(p\).

We also use weighted homomorphism densities. For a graph \(F\)
and vertex weights \(f_v\colon\Omega\to\R\), define
\begin{equation}\label{eq:weightedcount}
 t(F,K;\mathbf f)=
 \int_{\Omega^{V(F)}}
 \prod_{uv\in E(F)}K(x_u,x_v)
 \prod_{v\in V(F)}f_v(x_v)\,d\mu^{V(F)}(\mathbf x).
\end{equation}
The variables \(x_v\) are not required to be distinct.
When all vertex weights equal the same function \(f\), we
write \(t(F,K;f)\).

To express the hereditary counting error in this setting,
for a symmetric kernel \(W\colon\Omega^2\to[0,1]\) define
\[
 \Def_{H,p}(W)=
 \sup_{0\le f\le1}
 \left|t(H,W;f)-p^m\left(\int f\,d\mu\right)^r\right|.
\]
This quantity measures the counting error over all
\([0,1]\)-valued vertex weights, extending the tests given
by indicator functions of subsets.
All constants in the kernel estimates below are independent
of \(\Omega\) and \(\mu\).

\subsection{Polarization and averaged roots}\label{sec:polarization}

We first use polarization to pass from a common vertex weight
to averages over assignments of distinct vertex weights.
Identify \(V(H)\) with \([r]\), and define the symmetric
\(r\)-linear form
\[
\begin{aligned}
 B_W(f_1,\ldots,f_r)
 ={}&\frac1{r!}\sum_{\sigma\in\mathfrak S_r}
 \int_{\Omega^r}
 \prod_{uv\in E(H)}W(x_u,x_v)
 \prod_{i=1}^r f_{\sigma(i)}(x_i)\,d\mu^r(\mathbf x)\\
 &-p^m\prod_{i=1}^r\int f_i\,d\mu.
\end{aligned}
\]
On the diagonal, this reduces to
\(P_W(f):=B_W(f,\ldots,f)
=t(H,W;f)-p^m(\int f\,d\mu)^r\).
Thus \(\Def_{H,p}(W)\) controls the diagonal values of \(B_W\)
on \([0,1]\)-valued functions. The following polarization
estimate extends this control to arbitrary tuples of such
functions.

\begin{lemma}[\rev{Polarization; {\cite[Appendix, (A.1)]{Thomas}}}]\label{lem:polarization}
Let \(B_W\) and \(P_W\) be as above, and put
\(b_r=\frac1{r!}\sum_{k=1}^r\binom rk k^r\).

If \(\sup_{0\le f\le1}|P_W(f)|\le\xi\), then
\[
 |B_W(f_1,\ldots,f_r)|\le b_r\xi
 \qquad\text{whenever }0\le f_i\le1\text{ for all }i\in[r].
\]
\end{lemma}

\begin{proof}
Fix functions \(f_1,\ldots,f_r\) with \(0\le f_i\le1\).
Since \(B_W\) is symmetric and \(r\)-linear, its diagonal
polynomial satisfies \rev{the classical polarization identity~\cite[Appendix, (A.1)]{Thomas}}
\[
 r!B_W(f_1,\ldots,f_r)
 =\sum_{I\subseteq[r]}(-1)^{r-|I|}
 P_W\Bigl(\sum_{i\in I}f_i\Bigr).
\]
The term corresponding to \(I=\varnothing\) is zero because
\(P_W(0)=0\). For each nonempty \(I\subseteq[r]\), write
\(k=|I|\) and \(g_I=k^{-1}\sum_{i\in I}f_i\).
Then \(0\le g_I\le1\) pointwise, so the hypothesis and
\(r\)-homogeneity give
\(\bigl|P_W(\sum_{i\in I}f_i)\bigr|
=k^r|P_W(g_I)|\le k^r\xi\).
Applying the triangle inequality and grouping subsets by
their cardinality, we obtain
\[
 |B_W(f_1,\ldots,f_r)|
 \le \frac{\xi}{r!}
      \sum_{\varnothing\ne I\subseteq[r]}|I|^r
 = \frac{\xi}{r!}\sum_{k=1}^r\binom rk k^r
 = b_r\xi.
 \qedhere
\]
\end{proof}

We now apply Lemma~\ref{lem:polarization} to averaged rooted
densities. For \(v\in V(H)\), let \(t_v(H,W;x)\) denote the
integral in \eqref{eq:weightedcount} with all vertex weights
equal to one, with \(x_v=x\) fixed, and with all remaining
variables integrated out. For distinct \(v,w\in V(H)\),
define \(t_{v,w}(H,W;x,y)\) similarly by fixing
\(x_v=x\) and \(x_w=y\). Set
\begin{align*}
 R_H(x)
 &=\frac1r\sum_{v\in V(H)}t_v(H,W;x)-p^m,\\
 M_H(x,y)
 &=\frac1{r(r-1)}
   \sum_{\substack{v,w\in V(H)\\v\ne w}}
   t_{v,w}(H,W;x,y)-p^m.
\end{align*}
The second sum is over ordered pairs of distinct vertices.
These averages arise by allowing one or two arguments of
\(B_W\) to vary while setting the remaining arguments equal
to one.

\begin{corollary}\label{cor:rooted}
If \(\Def_{H,p}(W)\le\xi\), then
\[
 \norm{R_H}{1}\le2b_r\xi
 \qquad\text{and}\qquad
 \cut{M_H}\le b_r\xi.
\]
\end{corollary}

\begin{proof}
In the symmetrization defining \(B_W\), a distinguished
weight occupies each vertex in exactly \((r-1)!\) permutations.
Similarly, two distinguished weights occupy each ordered
pair of distinct vertices in exactly \((r-2)!\) permutations.
Consequently,
\[
 \begin{aligned}
 B_W(f,\one,\ldots,\one)
 &=\int_\Omega f(x)R_H(x)\,d\mu(x),\\
 B_W(f,g,\one,\ldots,\one)
 &=\int_{\Omega^2}f(x)g(y)M_H(x,y)\,d\mu(x)d\mu(y).
 \end{aligned}
\]
Here the second expression has no remaining constant
arguments when \(r=2\).

Since \(\Def_{H,p}(W)\le\xi\), Lemma~\ref{lem:polarization}
bounds the absolute value of each expression by \(b_r\xi\)
whenever \(0\le f,g\le1\).
Taking \(f=\one_{\{R_H>0\}}\) and
\(f=\one_{\{R_H<0\}}\) in the first identity gives
\(\int (R_H)_+\,d\mu\le b_r\xi\) and
\(\int (R_H)_-\,d\mu\le b_r\xi\), respectively, where
\((R_H)_+=\max\{R_H,0\}\) and
\((R_H)_-=\max\{-R_H,0\}\).
Adding these inequalities yields
\(\norm{R_H}{1}\le2b_r\xi\).

Finally, taking \(f=\one_A\) and \(g=\one_B\) in the second
identity gives
\(\bigl|\int_{A\times B}M_H\,d\mu^2\bigr|\le b_r\xi\)
for every \(A,B\subseteq\Omega\).
Maximizing over \(A\) and \(B\) proves
\(\cut{M_H}\le b_r\xi\).
\end{proof}

Both averages are taken over all vertices of \(H\), including
isolated vertices. The arguments below use only these averaged
bounds; no separate bound for an individual root or an
individual ordered pair of roots is required.

\subsection{Signed-kernel estimates}\label{sec:kernels}

We next collect the operator estimates used to control
signed-kernel densities. Throughout this subsection, let
\(U\colon\Omega^2\to[-1,1]\) be symmetric, and write
\(T=T_U\) and \(d=\cut{U}\).
Products of kernels always denote composition:
\((KL)(x,y)=\int_\Omega K(x,z)L(z,y)\,d\mu(z)\).
In particular, \(U^j\) is the kernel of \(T^j\).
Powers of one-variable functions, such as \(a^k\), are
pointwise. For a kernel evaluated at specified arguments,
brackets indicate a scalar power; for example,
\([U^2(x,y)]^2\) is the square of the value \(U^2(x,y)\).

All traces are operator traces on \(L^2(\mu)\).
To make the normalization explicit, in the orthonormal basis
\(\{\mu(x)^{-1/2}\one_{\{x\}}:x\in\Omega\}\), the operator
\(T\) is represented by the real symmetric matrix
\(A_{xy}=\sqrt{\mu(x)}\,U(x,y)\sqrt{\mu(y)}\).
Thus \(\tr(T^j)=\tr(A^j)\).

For an operator \(S\) and \(1\le q<\infty\), its Schatten
\(q\)-norm is
\(\norm{S}{S_q}=(\sum_i s_i(S)^q)^{1/q}\), where \(s_i(S)\)
are the singular values of \(S\).
Since \(T\) is self-adjoint,
\(\norm{T}{S_4}^4=\tr(T^4)\).
We use the standard Schatten ideal property and
Schatten H\"older inequality; see Simon~\cite{Simon}.

The estimates below are uniform over the underlying finite
probability space and, whenever vertex weights occur, over
all weights of absolute value at most one.
We call a bound \emph{superlinear} if it is
\(O(d^{1+\eta})\) for some \(\eta>0\) as \(d\to0\), with a
uniform implied constant. The term refers to the dependence
on the cut norm, rather than to the number of edge factors.

We begin with basic operator bounds and a cut-norm estimate
for composition with \(U\).

\begin{lemma}\label{lem:operators}
Let \(U\colon\Omega^2\to[-1,1]\) be symmetric, and put
\(T=T_U\) and \(d=\cut{U}\). Then
\[
 \begin{aligned}
 \norm{T}{\infty\to1}&\le4d,
 &\qquad \norm{T}{1\to\infty}&\le1,\\
 \norm{T}{1\to1}&\le1,
 &\qquad \norm{T}{\infty\to\infty}&\le1,
 \end{aligned}
\]
and
\[
 \norm{T}{2\to2}^{\,2}\le4d
 \qquad\text{and}\qquad
 \norm{T^3}{\infty\to1}\le16d^2.
\]
Moreover, \(\cut{KU}\le4\cut{K}\) for every real kernel \(K\),
without any symmetry assumption on \(K\).
\end{lemma}

\begin{proof}
The bound \(\norm{T}{\infty\to1}\le4d\) follows from
\eqref{eq:cutdual}.
Since \(|U|\le1\), we have \(|Tf(x)|\le\norm{f}{1}\)
for every \(x\in\Omega\).
Together with \(\mu(\Omega)=1\), this gives the other three
bounds in the first display.

To obtain the \(L^2\) bound, first observe that
\[
 \norm{T^2}{1\to1}
 \le \norm{T}{\infty\to1}\norm{T}{1\to\infty}
 \le4d.
\]
If \(Tv=\lambda v\) with \(v\ne0\), then
\(|\lambda|^2\norm{v}{1}
=\norm{T^2v}{1}\le4d\norm{v}{1}\), so
\(|\lambda|^2\le4d\).
The space is finite-dimensional, and symmetry of \(U\)
makes \(T\) self-adjoint on \(L^2(\mu)\).
Its \(L^2\) operator norm is therefore the largest absolute
value of its eigenvalues, proving
\(\norm{T}{2\to2}^{\,2}\le4d\).

For the cubic estimate, compose the bounds along
\(L^\infty\to L^1\to L^\infty\to L^1\) to obtain
\[
 \norm{T^3}{\infty\to1}
 \le
 \norm{T}{\infty\to1}
 \norm{T}{1\to\infty}
 \norm{T}{\infty\to1}
 \le16d^2.
\]

Finally, let \(K\) be any real kernel and fix
\(A,B\subseteq\Omega\). By the definition of kernel
composition and Fubini's theorem,
\[
 \int_{A\times B}(KU)(x,y)\,d\mu(x)d\mu(y)
 =
 \int_{\Omega^2}
 \one_A(x)K(x,z)(T\one_B)(z)\,d\mu(x)d\mu(z).
\]
Since \(\norm{\one_A}{\infty}\le1\) and
\(\norm{T\one_B}{\infty}\le1\), \eqref{eq:cutdual}
bounds the absolute value of this integral by
\(4\cut{K}\).
Maximizing over \(A\) and \(B\) yields
\(\cut{KU}\le4\cut{K}\).
\end{proof}

Lemma~\ref{lem:operators}, together with
\(\norm{T}{S_2}^2=\int_{\Omega^2}|U(x,y)|^2\,d\mu^2\le1\),
gives the linear bound
\(\tr(T^4)\le\norm{T}{2\to2}^2\norm{T}{S_2}^2\le4d\).
For the counting estimates below, we need a superlinear
bound in \(d\). \rev{The following consequence of the real Grothendieck
inequality, in the form stated in~\cite[Section~1]{AN}, provides the
required improvement. We include the argument to track the
exponent and the normalization used here.}

\begin{lemma}\label{lem:fourthtrace}
Let \(K_G\) denote the real Grothendieck constant. Then
\[
 0\le\tr(T^4)\le8K_Gd^{3/2}.
\]
\end{lemma}

\begin{proof}
Since \(T\) is self-adjoint,
\(\tr(T^4)=\norm{T^2}{S_2}^2\ge0\).
Put \(\rho=\norm{T}{2\to2}\).
If \(\rho=0\), then \(T=0\) and the conclusion is immediate,
so assume that \(\rho>0\).

For each \(x\in\Omega\), let \(u_x=U(x,\cdot)\), regarded
as a vector in the real Hilbert space \(L^2(\mu)\).
The assumptions \(|U|\le1\) and \(\mu(\Omega)=1\) imply
\(\norm{u_x}{2}\le1\), and hence
\(\norm{\rho^{-1}Tu_x}{2}\le1\).
By symmetry of \(U\) and the definition of kernel composition,
\(\langle u_x,Tu_y\rangle=U^3(x,y)\).
The trace formula therefore gives
\[
 \tr(T^4)
 =\int_{\Omega^2}U(x,y)U^3(y,x)\,d\mu(x)d\mu(y)
 =\int_{\Omega^2}
   U(x,y)\langle u_x,Tu_y\rangle\,d\mu(x)d\mu(y).
\]

Apply the real Grothendieck inequality~\rev{\cite[Section~1]{AN}}, in its
unit-ball form, to the matrix
\(c_{xy}=\mu(x)\mu(y)U(x,y)\) and the two families of vectors
\(u_x\) and \(\rho^{-1}Tu_y\).
The unit-ball form follows from the unit-vector form by
adjoining orthogonal coordinates without changing the
inner products between the two families.
We obtain
\[
 \tr(T^4)
 \le K_G\rho
 \max_{\varepsilon,\eta\colon\Omega\to\{-1,1\}}
 \left|\sum_{x,y\in\Omega}
 c_{xy}\varepsilon(x)\eta(y)\right|.
\]
The maximum on the right equals
\(\norm{T}{\infty\to1}\): by \(L^1\)--\(L^\infty\) duality,
this norm is the maximum of the same bilinear form over
\([-1,1]^\Omega\times[-1,1]^\Omega\), and a maximum absolute
value is attained at vertices of these cubes.

Finally, Lemma~\ref{lem:operators} gives
\(\norm{T}{\infty\to1}\le4d\) and \(\rho\le2d^{1/2}\).
Consequently,
\[
 \tr(T^4)
 \le K_G\rho\norm{T}{\infty\to1}
 \le4K_G\rho d
 \le8K_Gd^{3/2}.
 \qedhere
\]
\end{proof}

We now use these operator bounds to estimate weighted
densities of graphs containing certain induced subgraphs.
The estimates allow signed vertex weights: after fixing
the variables outside an induced subgraph, the incident
edge factors contribute such weights to its vertices.

\begin{lemma}\label{lem:patterns}
Let \(F\) be a graph, and let
\(f_v\colon\Omega\to[-1,1]\) for each \(v\in V(F)\).
If \(F\) contains an induced copy of \(P_4\) or \(2K_2\), then
\[
 |t(F,U;\mathbf f)|\le16d^2.
\]
If \(F\) contains an induced copy of \(C_4\), then
\[
 |t(F,U;\mathbf f)|\le8K_Gd^{3/2}.
\]
\end{lemma}

\begin{proof}
We first reduce to the three stated graphs.
Let \(Q\subseteq V(F)\) induce one of them, and put
\(R=V(F)\setminus Q\).
For fixed external variables
\(\mathbf z=(z_w)_{w\in R}\in\Omega^R\), define
\[
 \begin{aligned}
 g_v^{\mathbf z}(x)
 &=f_v(x)
   \prod_{\substack{w\in R\\vw\in E(F)}}U(x,z_w)
   \qquad(v\in Q),\\
 c(\mathbf z)
 &=\prod_{w\in R}f_w(z_w)
   \prod_{vw\in E(F[R])}U(z_v,z_w).
 \end{aligned}
\]
All these factors have absolute value at most one.
Partitioning the edge factors according to their endpoints
and applying Fubini's theorem gives
\[
 t(F,U;\mathbf f)
 =
 \int_{\Omega^R}
 c(\mathbf z)\,
 t(F[Q],U;\mathbf g^{\mathbf z})\,d\mu^R(\mathbf z).
\]
The induced-subgraph assumption ensures that the internal
edge factors are exactly those of the chosen pattern.
Thus any bound for its weighted density that is uniform
over vertex weights of absolute value at most one also
bounds \(t(F,U;\mathbf f)\).

For \(P_4\), label the vertices consecutively by \(1,2,3,4\).
Integrating over the endpoints and using symmetry of \(U\),
we obtain
\[
 t(P_4,U;\mathbf g)
 =
 \int_{\Omega^2}
 g_2(x)g_3(y)U(x,y)
 (Tg_1)(x)(Tg_4)(y)\,d\mu(x)d\mu(y).
\]
Since \(|g_2|,|g_3|,|U|\le1\), Lemma~\ref{lem:operators}
gives
\[
 |t(P_4,U;\mathbf g)|
 \le\norm{Tg_1}{1}\norm{Tg_4}{1}
 \le\norm{T}{\infty\to1}^{\,2}
 \le16d^2.
\]

For \(2K_2\), the density factors into two weighted edge
integrals. Each has absolute value at most
\(\norm{T}{\infty\to1}\le4d\), since
\(\bigl|\int g_i(Tg_j)\,d\mu\bigr|
\le\norm{g_i}{\infty}\norm{Tg_j}{1}\).
Their product is therefore at most \(16d^2\).

Finally, label the vertices of \(C_4\) in cyclic order by
\(1,2,3,4\), and let \(D_i\) be multiplication by \(g_i\)
on \(L^2(\mu)\).
Then \(\norm{D_i}{2\to2}\le1\), and the trace formula gives
\(t(C_4,U;\mathbf g)=\tr(D_1TD_2TD_3TD_4T)\).
The Schatten ideal property yields
\(\norm{D_iT}{S_4}\le\norm{T}{S_4}\).
Applying Schatten H\"older to the four factors \(D_iT\),
and then using self-adjointness of \(T\) and
Lemma~\ref{lem:fourthtrace}, we conclude that
\[
 \begin{aligned}
 |t(C_4,U;\mathbf g)|
 &\le\prod_{i=1}^4\norm{D_iT}{S_4}\\
 &\le\norm{T}{S_4}^4
 =\tr(T^4)
 \le8K_Gd^{3/2}.
 \end{aligned}
 \qedhere
\]
\end{proof}

The pattern estimates above yield a uniform bound for
triangle-free graphs outside a single exceptional family.
We call a graph a \emph{star with isolated vertices} if its
nonisolated part is \(K_{1,k}\) for some \(k\ge1\).

\begin{lemma}\label{lem:nonstar}
There exists an absolute constant \(C_0>0\) such that the
following holds. Let \(F\) be a triangle-free graph with
at least one edge, and suppose that \(F\) is not a star
with isolated vertices. Then, for all vertex weights
\(f_v\colon\Omega\to[-1,1]\),
\[
 |t(F,U;\mathbf f)|\le C_0d^{3/2}.
\]
Moreover, every one-root density of \(F\) in \(U\) has
\(L^1\) norm at most \(C_0d^{3/2}\), and every two-root
density, with distinct roots, has cut norm at most
\(C_0d^{3/2}\).

If \(F\) also contains no four-cycle, each of these three
bounds improves to \(16d^2\).
\end{lemma}

\begin{proof}
We first show that \(F\) contains an induced copy of
\(P_4\), \(2K_2\), or \(C_4\).
Suppose that \(F\) contains a cycle, and choose a shortest
one. This cycle is induced and, since \(F\) is triangle-free,
has length at least four. If its length is four, it gives
an induced \(C_4\); otherwise, four consecutive vertices
induce a \(P_4\).

Now suppose that \(F\) is a forest. If two components
contain edges, choosing one edge from each gives an
induced \(2K_2\). Otherwise, the nonisolated part of \(F\)
is a single tree. This tree is not a star, so its diameter
is at least three, and four consecutive vertices on a
longest path induce a \(P_4\).
If \(F\) contains no four-cycle, the same argument always
produces an induced \(P_4\) or \(2K_2\).

Since \(|U|\le1\) on a probability space, we have
\(0\le d\le1\), and hence \(d^2\le d^{3/2}\).
Lemma~\ref{lem:patterns} therefore gives the weighted
density bound with \(C_0=\max\{16,8K_G\}\).
When \(F\) contains no four-cycle, it gives the stronger
bound \(16d^2\).

To deduce the rooted estimates, write \(B=C_0d^{3/2}\),
or \(B=16d^2\) in the four-cycle-free case.
Fix a vertex \(v\), and let \(h(x)=t_v(F,U;x)\).
Choose \(f_v=\operatorname{sgn}(h)\), with
\(\operatorname{sgn}(0)=0\), and set all other vertex
weights equal to one. The weighted density bound gives
\[
 \norm{h}{1}
 =\int_\Omega \operatorname{sgn}(h(x))h(x)\,d\mu(x)
 =t(F,U;\mathbf f)
 \le B.
\]

For distinct vertices \(v,w\), put
\(k(x,y)=t_{v,w}(F,U;x,y)\).
Given arbitrary \(A,B'\subseteq\Omega\), take
\(f_v=\one_A\), \(f_w=\one_{B'}\), and all remaining
weights equal to one. Then
\[
 \left|\int_{A\times B'}k(x,y)\,d\mu(x)d\mu(y)\right|
 =|t(F,U;\mathbf f)|
 \le B.
\]
Maximizing over \(A\) and \(B'\) yields \(\cut{k}\le B\),
as required.
\end{proof}

To extend the argument beyond triangle-free graphs, we
need estimates for two further induced patterns.
The \emph{paw} is obtained from a triangle by adding a
pendant vertex, and the \emph{diamond} is obtained from
\(K_4\) by deleting one edge.

\begin{lemma}\label{lem:pawdiamond}
Let \(F\) be a graph, and let
\(f_v\colon\Omega\to[-1,1]\) for each \(v\in V(F)\).
If \(F\) contains an induced paw, then
\[
 |t(F,U;\mathbf f)|\le2\sqrt{8K_G}\,d^{5/4}.
\]
If \(F\) contains an induced diamond, then
\[
 |t(F,U;\mathbf f)|\le8K_Gd^{3/2}.
\]
\end{lemma}

\begin{proof}
By the conditioning argument in the proof of
Lemma~\ref{lem:patterns}, it suffices to estimate the two
patterns themselves with arbitrary vertex weights
\(|g_i|\le1\).
Indeed, fixing the external variables absorbs all crossing
edge factors into the internal vertex weights, while the
product of the external vertex weights and external edge
factors has absolute value at most one.
A uniform bound for the internal integral therefore
survives integration over the external variables.
Write \(D_i\) for multiplication by \(g_i\) on \(L^2(\mu)\);
then \(\norm{D_i}{2\to2}\le1\).

For the paw, label the triangle \(123\) and the pendant
edge \(14\).
Integrating over vertex \(4\) replaces \(g_1\) by
\(h=g_1(Tg_4)\).
Since \(\norm{g_4}{2}\le1\), Lemma~\ref{lem:operators}
gives
\(\norm{h}{2}\le\norm{Tg_4}{2}
\le\norm{T}{2\to2}\le2d^{1/2}\).
Let \(D_h\) denote multiplication by \(h\).
The remaining weighted triangle integral is
\(t(\mathrm{paw},U;\mathbf g)=\tr(D_hTD_2TD_3T)\).
Moreover, the Hilbert--Schmidt kernel formula and
\(|U|\le1\) give
\[
 \norm{D_hT}{S_2}^2
 =\int_{\Omega^2}|h(x)|^2|U(x,y)|^2\,d\mu(x)d\mu(y)
 \le\norm{h}{2}^2
 \le4d.
\]
For \(i=2,3\), the Schatten ideal property yields
\(\norm{D_iT}{S_4}\le\norm{T}{S_4}\).
Applying Schatten H\"older with exponents \(2,4,4\),
followed by Lemma~\ref{lem:fourthtrace}, we obtain
\[
 \begin{aligned}
 |t(\mathrm{paw},U;\mathbf g)|
 &\le\norm{D_hT}{S_2}
      \norm{D_2T}{S_4}\norm{D_3T}{S_4}\\
 &\le2d^{1/2}\norm{T}{S_4}^2
 =2d^{1/2}\tr(T^4)^{1/2}
 \le2\sqrt{8K_G}\,d^{5/4}.
 \end{aligned}
\]

For the diamond, let \(12\) be the edge shared by its
two triangles, and let \(3,4\) be the nonadjacent vertices.
For \(i=3,4\), define
\(Q_i(x,y)=\int_\Omega U(x,z)g_i(z)U(z,y)\,d\mu(z)\),
the kernel of \(TD_iT\).
By symmetry of \(U\), integrating over vertices \(3\)
and \(4\) gives
\[
 t(\mathrm{diamond},U;\mathbf g)
 =
 \int_{\Omega^2}
 g_1(x)g_2(y)U(x,y)Q_3(x,y)Q_4(x,y)\,d\mu(x)d\mu(y).
\]
Since \(|g_1|,|g_2|,|U|\le1\), Cauchy--Schwarz bounds
the absolute value of this integral by
\(\norm{Q_3}{L^2(\mu^2)}\norm{Q_4}{L^2(\mu^2)}
=\norm{TD_3T}{S_2}\norm{TD_4T}{S_2}\).
For each \(i=3,4\), Schatten H\"older and the ideal property
give
\(\norm{TD_iT}{S_2}
\le\norm{T}{S_4}\norm{D_iT}{S_4}
\le\norm{T}{S_4}^2\).
Hence Lemma~\ref{lem:fourthtrace} implies
\[
 |t(\mathrm{diamond},U;\mathbf g)|
 \le\norm{T}{S_4}^4
 =\tr(T^4)
 \le8K_Gd^{3/2}.
 \qedhere
\]
\end{proof}

We call a graph a \emph{clique with isolated vertices} if
its nonisolated part is \(K_q\) for some \(q\ge2\).
The following structural observation shows that the five
patterns considered above cover every graph outside the
two families of stars and cliques with isolated vertices.
Combining their density estimates therefore gives a
uniform superlinear bound for all remaining graphs.

\begin{lemma}\label{lem:fivepatterns}
Let \(F\) be a graph with at least one edge that is neither
a star with isolated vertices nor a clique with isolated
vertices. Then \(F\) contains an induced copy of at least
one of \(P_4\), \(2K_2\), \(C_4\), the paw, or the diamond.

Consequently, there exists an absolute constant \(C_*>0\)
such that, for every such \(F\) and all vertex weights
\(f_v\colon\Omega\to[-1,1]\),
\[
 |t(F,U;\mathbf f)|\le C_*d^{5/4}.
\]
Moreover, every two-root density of \(F\), with distinct
roots, has cut norm at most \(C_*d^{5/4}\).
\end{lemma}

\begin{proof}
We prove the structural assertion by contraposition.
Suppose that \(F\) contains none of the five listed
induced subgraphs. Since two edge-containing components
would give an induced \(2K_2\), exactly one component
\(C\) contains an edge, and all other vertices are isolated.
We show that \(C\) is a star or a clique.

\rev{By Wolk's universal-vertex lemma, a connected graph with no
induced \(P_4\) or \(C_4\) has a universal vertex;
see~\cite[Lemma~4]{Andres}. Choose such a vertex \(v\) of \(C\).}

If \(C-v\) is independent, then \(C\) is a star.
Otherwise, choose an edge \(ab\) in \(C-v\).
For any \(w\in V(C)\setminus\{v,a,b\}\), the four vertices
\(v,a,b,w\) induce a paw if \(w\) is adjacent to neither
\(a\) nor \(b\), and a diamond if it is adjacent to exactly
one of them.
Thus every such \(w\) is adjacent to both \(a\) and \(b\).
If two vertices \(w,z\in V(C)\setminus\{v,a,b\}\) were
nonadjacent, then \(a,b,w,z\) would induce a diamond.
Hence \(C-v\) is complete, and so is \(C\).
This proves the structural assertion.

For the density estimate, set
\(C_*=\max\{16,8K_G,2\sqrt{8K_G}\}\).
By the structural assertion, one of
Lemmas~\ref{lem:patterns} and~\ref{lem:pawdiamond} applies.
Since \(0\le d\le1\), both \(d^2\) and \(d^{3/2}\) are
at most \(d^{5/4}\). Each of the five pattern estimates
therefore yields
\(\lvert t(F,U;\mathbf f)\rvert\le C_*d^{5/4}\).

Finally, fix any distinct vertices \(u,v\in V(F)\), and
write \(k(x,y)=t_{u,v}(F,U;x,y)\).
For arbitrary \(A,B\subseteq\Omega\), choose vertex weights
\(f_u=\one_A\), \(f_v=\one_B\), and \(f_w=\one\) for
\(w\notin\{u,v\}\).
The weighted estimate gives
\[
 \left|\int_{A\times B}k(x,y)\,d\mu(x)d\mu(y)\right|
 =|t(F,U;\mathbf f)|
 \le C_*d^{5/4}.
\]
Maximizing over \(A\) and \(B\) proves
\(\cut{k}\le C_*d^{5/4}\).
\end{proof}

\section{The general local estimate}\label{sec:general}

We now establish the local kernel estimate underlying
Theorem~\ref{thm:main}. It bounds the cut-norm discrepancy
by the hereditary defect together with a superlinear
remainder, which can be absorbed when the discrepancy
is sufficiently small.

\begin{proposition}\label{prop:general-local}
Let \(H\) be a fixed graph with \(m\ge1\) edges, and let
\(p\in(0,1)\). There exists a constant \(L(H,p)\) with
\(1\le L(H,p)\le c_Hp^{-6m}\), where \(c_H\) depends only
on \(H\), such that every symmetric kernel
\(W\colon\Omega^2\to[0,1]\) on a finite probability space
satisfies
\[
 \cut{W-p}\le L(H,p)
 \bigl(\Def_{H,p}(W)+\cut{W-p}^{5/4}\bigr).
\]
\end{proposition}

Throughout this section, write \(r=v(H)\), \(m=e(H)\),
\(U=W-p\), and \(T=T_U\). We abbreviate the discrepancy
and hereditary defect by \(d=\cut{U}\) and
\(\xi=\Def_{H,p}(W)\), respectively.
We also set \(a=T\one\) and \(A_2=\int_\Omega a^2\,d\mu\),
so that \(a\) is the degree deviation and \(A_2\) measures
its squared \(L^2\) norm.

For each \(J\subseteq E(H)\), let \(F_J=(V(H),J)\) be
the spanning subgraph with edge set \(J\); in particular,
all isolated vertices are retained.
Constants denoted by \(C_H\) depend only on \(H\) and may
vary from one occurrence to the next.

\subsection{Degree variance}

We first use the averaged one-root density to control
\(A_2\), the squared \(L^2\) norm of the degree deviation.
This estimate will later bound the contributions of
large stars arising from edge substitution.

\begin{lemma}\label{lem:generaldegree}
With the notation above, there exists a constant \(C_H>0\),
depending only on \(H\), such that
\[
 A_2\le C_Hp^{-2(m-1)}(\xi+d^2).
\]
\end{lemma}

\begin{proof}
Since \(a(x)=\int_\Omega W(x,y)\,d\mu(y)-p\), we have
\(-p\le a(x)\le1-p\), and in particular
\(\norm{a}{\infty}\le1\).

For \(v\in V(H)\), let \(E_H(v)\) be the set of edges
incident with \(v\).
Expand each rooted density in \(R_H\) using \(W=p+U\).
The empty edge set contributes \(p^m\), which cancels
the constant term in \(R_H\).

If \(\varnothing\ne J\subseteq E_H(v)\), then
\(t_v(F_J,U;x)=a(x)^{|J|}\): each leaf variable contributes
a factor \(a(x)\), and every isolated vertex contributes
one.
Summing these terms over all roots gives
\[
 \Phi_H(t)=
 \sum_{\substack{v\in V(H)\\\deg_H(v)>0}}
 p^{m-\deg_H(v)}
 \bigl((p+t)^{\deg_H(v)}-p^{\deg_H(v)}\bigr).
\]
Thus the full expansion is
\begin{equation}\label{eq:generalrootidentity}
 rR_H=\Phi_H(a)+\mathcal Q,
\end{equation}
where
\[
 \mathcal Q=
 \sum_{v\in V(H)}
 \sum_{\substack{\varnothing\ne J\subseteq E(H)\\
                  J\nsubseteq E_H(v)}}
 p^{m-|J|}t_v(F_J,U;\cdot).
\]

Let \(r_+\) be the number of nonisolated vertices of \(H\).
For every integer \(k\ge1\) and every \(t\ge-p\),
\[
 t\bigl((p+t)^k-p^k\bigr)
 =t^2\sum_{j=0}^{k-1}(p+t)^{k-1-j}p^j
 \ge p^{k-1}t^2,
\]
because all terms in the sum are nonnegative and its last
term is \(p^{k-1}\).
Applying this inequality to each summand of \(\Phi_H\)
yields
\begin{equation}\label{eq:coercivity}
 t\Phi_H(t)\ge r_+p^{m-1}t^2
 \qquad(t\ge-p).
\end{equation}
Since \(a\ge-p\), integration gives
\begin{equation}\label{eq:generalcoercive}
 \int_\Omega a\Phi_H(a)\,d\mu
 \ge r_+p^{m-1}A_2.
\end{equation}

We next bound \(\mathcal Q\) in \(L^2(\mu)\).
Consider a rooted density \(q=t_v(F_J,U;\cdot)\) occurring
in \(\mathcal Q\).
The condition \(J\nsubseteq E_H(v)\) means that
\(F_J-v\) contains an edge.
Form a graph \(D\) from two disjoint copies of \(F_J\)
by identifying their copies of \(v\), with no other
identifications.
This graph is simple and contains an induced \(2K_2\):
choose an edge avoiding the root in each copy.
Their four endpoints are distinct, and there are no edges
between the two copies away from the common root.

By Fubini's theorem, with all vertex weights equal to one,
\[
 \norm{q}{2}^2
 =\int_\Omega q(x)^2\,d\mu(x)
 =t(D,U;\one)
 \le16d^2,
\]
where the last inequality follows from
Lemma~\ref{lem:patterns}.
Hence \(\norm{q}{2}\le4d\).
The number of terms in \(\mathcal Q\) depends only on \(H\),
and every coefficient satisfies \(p^{m-|J|}\le1\).
The triangle inequality therefore gives
\(\norm{\mathcal Q}{2}\le C_Hd\).

Multiplying \eqref{eq:generalrootidentity} by \(a\),
integrating, and using \eqref{eq:generalcoercive}, we obtain
\[
 \begin{aligned}
 r_+p^{m-1}A_2
 &\le r\left|\int_\Omega aR_H\,d\mu\right|
      +\left|\int_\Omega a\mathcal Q\,d\mu\right|\\
 &\le r\norm{a}{\infty}\norm{R_H}{1}
      +\norm{a}{2}\norm{\mathcal Q}{2}\\
 &\le2rb_r\xi+C_Hd\sqrt{A_2}.
 \end{aligned}
\]
Here we used Corollary~\ref{cor:rooted} and
Cauchy--Schwarz.

Since \(H\) has an edge, \(r_+\ge2\).
Young's inequality gives
\[
 C_Hd\sqrt{A_2}
 \le\frac12r_+p^{m-1}A_2
     +C_Hp^{-(m-1)}d^2.
\]
Absorbing the first term and dividing by \(p^{m-1}\),
with factors depending only on \(H\) incorporated into
\(C_H\), yields
\[
 A_2
 \le C_H\bigl(p^{-(m-1)}\xi+p^{-2(m-1)}d^2\bigr)
 \le C_Hp^{-2(m-1)}(\xi+d^2),
\]
where the final inequality uses \(0<p<1\) and \(m\ge1\).
\end{proof}

\subsection{A selected edge and the marginal maps}

We next express rooted densities as linear maps in a
selected edge factor, keeping all other edge factors fixed.
Let \(F\) be a simple graph with distinct roots
\(b,c\in V(F)\), and let \(e=(i,j)\) be an oriented edge
of \(F\). Write \(\underline e=\{i,j\}\) for its underlying
undirected edge. With \(x_b=x\) and \(x_c=y\) fixed, define
\[
 (\mathcal B_{F;b,c;e}^U X)(x,y)
 =
 \int_{\Omega^{V(F)\setminus\{b,c\}}}
 X(x_i,x_j)
 \prod_{uv\in E(F)\setminus\{\underline e\}}U(x_u,x_v)
 \,d\mu^{V(F)\setminus\{b,c\}}.
\]
This map is linear in \(X\), while every other edge factor
remains equal to \(U\). In particular,
\(\mathcal B_{F;b,c;e}^U U=t_{b,c}(F,U;\cdot,\cdot)\).
The representation is exact: it involves neither
differentiation nor truncation of an expansion.
The orientation of \(e\) specifies the order of the
arguments of \(X\), which need not be symmetric.

To handle the one-variable and constant terms in rooted
expansions, we also introduce the marginal maps
\[
 \begin{aligned}
 (\mathcal R_xK)(x,y)&=\int_\Omega K(x,z)\,d\mu(z),\\
 (\mathcal R_yK)(x,y)&=\int_\Omega K(z,y)\,d\mu(z),
 \end{aligned}
\]
and the averaging map
\((\mathcal JK)(x,y)
=\int_{\Omega^2}K(z,w)\,d\mu(z)d\mu(w)\).
Thus \(\mathcal R_x\) retains the first variable,
\(\mathcal R_y\) retains the second, and \(\mathcal J\)
produces a constant kernel.

The following bounds control these maps in the cut norm.
Their interpretation in terms of rooted graphs will also
be used in the edge-substitution argument.

\begin{lemma}\label{lem:mapbounds}
For every real kernel \(X\),
\[
 \cut{\mathcal B_{F;b,c;e}^U X}\le4\cut{X}.
\]
Each of \(\mathcal R_x,\mathcal R_y,\mathcal J\) has
operator norm at most one with respect to the cut norm.

Moreover, applying any of these three maps to a two-root
density yields another two-root density, represented by
a simple graph obtained by retaining all original vertices
and edges and adding isolated vertices as new roots.
\end{lemma}

\begin{proof}
Write \(e=(i,j)\), and fix \(A,B\subseteq\Omega\).
By the definition of the selected-edge map and Fubini's
theorem,
\[
 \begin{aligned}
 &\int_{A\times B}
   (\mathcal B_{F;b,c;e}^U X)(x,y)\,d\mu(x)d\mu(y)\\
 &\qquad=
 \int_{\Omega^{V(F)}}
 \one_A(x_b)\one_B(x_c)X(x_i,x_j)
 \prod_{uv\in E(F)\setminus\{\underline e\}}U(x_u,x_v)
 \,d\mu^{V(F)}(\mathbf x).
 \end{aligned}
\]
Fix all variables except \(x_i,x_j\).
Because \(F\) is simple and the edge \(\underline e\)
has been removed from the product, every remaining
factor depends on at most one of these two variables.
The integrand therefore has the form
\(c\,f(x_i)X(x_i,x_j)g(x_j)\), where
\(|c|\le1\), \(\norm{f}{\infty}\le1\), and
\(\norm{g}{\infty}\le1\).
If a root is an endpoint of \(e\), its indicator is
included in the corresponding function \(f\) or \(g\);
otherwise, its indicator is included in \(c\).

By \eqref{eq:cutdual}, the absolute value of the integral
over \(x_i,x_j\) is at most \(4\cut{X}\).
This bound is uniform in the fixed variables, so integrating
over them and maximizing over \(A,B\) proves the first
assertion.

For the marginal maps, let \(K\) be any real kernel.
For every \(A,B\subseteq\Omega\), Fubini's theorem gives
\[
 \begin{aligned}
 \left|\int_{A\times B}\mathcal R_xK\,d\mu^2\right|
 &=\mu(B)\left|\int_{A\times\Omega}K\,d\mu^2\right|
 \le\cut{K},\\
 \left|\int_{A\times B}\mathcal R_yK\,d\mu^2\right|
 &=\mu(A)\left|\int_{\Omega\times B}K\,d\mu^2\right|
 \le\cut{K},\\
 \left|\int_{A\times B}\mathcal JK\,d\mu^2\right|
 &=\mu(A)\mu(B)\left|\int_{\Omega^2}K\,d\mu^2\right|
 \le\cut{K}.
 \end{aligned}
\]
Taking the maximum over \(A,B\) proves that all three
maps are contractions in the cut norm.

Finally, let \(K=t_{b,c}(F,U;\cdot,\cdot)\).
Integrating out a root variable makes that vertex an
ordinary integrated vertex.
The resulting kernel can still be represented as a
two-root density by introducing a new isolated root
for each missing free variable.
The constructions are summarized below, where \(I\)
denotes the identity map and \(b',c'\) are distinct new
vertices.
\begin{center}
\begin{tabular}{@{}lll@{}}
\toprule
Map & Ordered roots & Isolated vertices added\\
\midrule
\(I\)             & \(b,c\)   & none\\
\(\mathcal R_x\)  & \(b,c'\)  & \(c'\)\\
\(\mathcal R_y\)  & \(b',c\)  & \(b'\)\\
\(\mathcal J\)    & \(b',c'\) & \(b',c'\)\\
\bottomrule
\end{tabular}
\end{center}
In each case, the old roots remain vertices of the graph,
whether or not they remain distinguished.
No original vertices are deleted or identified, and no
edges are added or removed.
The resulting graph is therefore simple and has exactly
the original edge set.
\end{proof}

Before composing the selected-edge maps, we isolate the star
case that can remain after substitution.

\begin{lemma}\label{lem:largestars}
Let \(F\) be a star with \(k\ge3\) edges together with
isolated vertices. For any two distinct roots \(b,c\),
\[
 \cut{t_{b,c}(F,U;\cdot,\cdot)}
 \le A_2+2\sqrt{A_2d}.
\]
\end{lemma}

\begin{proof}
Fix \(A,B\subseteq\Omega\), and put \(g=\one_A\) and
\(h=\one_B\). We bound the integral of the two-root
density against \(g(x)h(y)\).
Recall that \(|a|\le1\) and \(\norm{a}{2}^2=A_2\).
By Lemma~\ref{lem:operators},
\(\norm{Tg}{\infty},\norm{Th}{\infty}\le1\) and
\(\norm{Tg}{2}\le2d^{1/2}\), since
\(\norm{g}{2}\le1\).
Every isolated vertex that is not a root contributes
one; an isolated root contributes either \(\int g\,d\mu\)
or \(\int h\,d\mu\), both at most one.

Suppose first that the two roots are not both leaves.
Then at least \(k-1\) leaves are not roots.
Integrating their variables produces a factor \(a^s\)
at the centre, where \(s\ge k-1\ge2\).
If a leaf is a root, integrating it against its indicator
contributes \(Tg\) or \(Th\), whose absolute value is at
most one. If the centre is a root, its indicator is also
bounded by one. Thus the absolute value of the cut
integral is at most
\(\int_\Omega |a|^s\,d\mu\le A_2\).

It remains to consider the case in which both roots are
leaves. Integrating the other \(k-2\) leaves and then
the two root variables gives the cut integral
\[
 \int_\Omega a^{k-2}(Tg)(Th)\,d\mu.
\]
If \(k\ge4\), its absolute value is at most
\(\int_\Omega |a|^{k-2}\,d\mu\le A_2\).
If \(k=3\), Cauchy--Schwarz gives
\[
 \left|\int_\Omega a(Tg)(Th)\,d\mu\right|
 \le\norm{a}{2}\norm{Tg}{2}\norm{Th}{\infty}
 \le2\sqrt{A_2d}.
\]
Every root placement therefore gives an absolute cut
integral at most \(A_2+2\sqrt{A_2d}\).
Maximizing over \(A\) and \(B\) proves the claim.
\end{proof}

The selected-edge maps admit a direct interpretation
in terms of graph substitution: inserting a two-root
density into the selected edge replaces that edge by
the rooted graph. The next lemma shows that the resulting
graph cannot be a clique with isolated vertices.
This is the key observation used to control clique
contributions after composing the correction map with itself.

\begin{lemma}[Edge substitution]\label{lem:substitution}
Let \(F_1,F_2\) be simple graphs, each with at least two
edges. Choose distinct roots \(b_1,c_1\in V(F_1)\),
an oriented edge \(e=(i,j)\) of \(F_1\), and distinct
roots \(b_2,c_2\in V(F_2)\). Then, after increasing the
absolute constant \(C_*\) if necessary,
\[
 \cut{
   \mathcal B_{F_1;b_1,c_1;e}^U
   \bigl(t_{b_2,c_2}(F_2,U;\cdot,\cdot)\bigr)
 }
 \le C_*\bigl(d^{5/4}+A_2+\sqrt{A_2d}\bigr).
\]
\end{lemma}

\begin{proof}
Take disjoint copies of \(F_1\) and \(F_2\).
Delete \(\underline e=\{i,j\}\) from \(F_1\), then identify
\(b_2\) with \(i\) and \(c_2\) with \(j\).
Let \(F\) be the resulting graph, retaining \(b_1,c_1\)
as its ordered roots.

No two vertices within either copy are identified.
Thus no loops are created, and the roots \(b_1,c_1\)
remain distinct, including when they coincide with
endpoints of \(e\).
The two copies intersect only at \(i,j\), so the only
possible duplicated edge is \(ij\).
Its copy in \(F_1\) has been deleted, and hence \(F\)
is simple. In particular,
\(e(F)=e(F_1)-1+e(F_2)\ge3\).

Expanding the inserted rooted density and applying
Fubini's theorem gives the exact kernel identity
\[
 \mathcal B_{F_1;b_1,c_1;e}^U
 \bigl(t_{b_2,c_2}(F_2,U;\cdot,\cdot)\bigr)
 =
 t_{b_1,c_1}(F,U;\cdot,\cdot).
\]
Indeed, the variables at \(b_2,c_2\) become those at
\(i,j\), respectively, while all other variables of
\(F_2\) remain separate integration variables.

We next show that \(F\) is not a clique with isolated
vertices. Since \(F_1\) has at least two edges,
\(F_1-\underline e\) contains an edge.
This edge has an endpoint \(x\notin\{i,j\}\), because
\(F_1\) is simple and its edge \(ij\) has been removed.
Likewise, \(F_2\) has an edge with an endpoint
\(y\notin\{b_2,c_2\}\), since at most one of its edges
can have both endpoints in \(\{b_2,c_2\}\).
The vertices \(x,y\) remain distinct and nonisolated
in \(F\). They lie outside the intersection of the
two copies, so they are nonadjacent.
Thus the nonisolated part of \(F\) is not complete.

If \(F\) is not a star with isolated vertices,
Lemma~\ref{lem:fivepatterns} gives
\[
 \cut{t_{b_1,c_1}(F,U;\cdot,\cdot)}
 \le C_*d^{5/4}.
\]
Otherwise, \(F\) is a star with at least three edges
together with isolated vertices, and
Lemma~\ref{lem:largestars} gives
\[
 \cut{t_{b_1,c_1}(F,U;\cdot,\cdot)}
 \le A_2+2\sqrt{A_2d}.
\]
Increasing \(C_*\) to at least \(2\) covers both cases
and proves the claimed bound.
\end{proof}

\subsection{Removing marginals and composing the correction}

The single-edge contribution to the averaged two-root
density involves \(U\), its one-variable marginals, and
its overall mean. We first invert this combination
explicitly, with a bound uniform in \(r\).

\begin{lemma}[Marginal inversion]\label{lem:marginalinverse}
For an integer \(r\ge2\), define the linear maps
\[
 \begin{aligned}
 \mathcal L_rX
 &=X+(r-2)(\mathcal R_xX+\mathcal R_yX)
   +\frac{(r-2)(r-3)}2\,\mathcal JX,\\
 \mathcal P_rK
 &=K-\frac{r-2}{r-1}(\mathcal R_xK+\mathcal R_yK)
   +\frac{r-2}{r}\,\mathcal JK.
 \end{aligned}
\]
Then \(\mathcal P_r\mathcal L_r=I\), where \(I\) is the
identity map on real kernels. Moreover, for every real
kernel \(K\),
\[
 \cut{\mathcal P_rK}\le4\cut{K}.
\]
\end{lemma}

\begin{proof}
Let \(X\) be any real kernel. Define
\(c=\int_{\Omega^2}X\,d\mu^2\),
\(f(x)=\int_\Omega X(x,y)\,d\mu(y)-c\), and
\(g(y)=\int_\Omega X(x,y)\,d\mu(x)-c\).
Then
\[
 X(x,y)=X_0(x,y)+f(x)+g(y)+c,
\]
where \(X_0=X-\mathcal R_xX-\mathcal R_yX+c\).
Both marginals of \(X_0\) vanish, while
\(\int_\Omega f\,d\mu=\int_\Omega g\,d\mu=0\).
This decomposes \(X\) into a kernel with zero marginals,
two mean-zero one-variable components, and a constant.

Each component is preserved by \(\mathcal L_r\) and
\(\mathcal P_r\). Substitution into their definitions gives
\[
 \begin{aligned}
 \mathcal L_rX
 &=X_0+(r-1)f(x)+(r-1)g(y)
   +\frac{r(r-1)}2\,c,\\
 \mathcal P_rX
 &=X_0+\frac{f(x)}{r-1}+\frac{g(y)}{r-1}
   +\frac{2}{r(r-1)}\,c.
 \end{aligned}
\]
In particular, the constant multiplier in the second
identity is
\(1-2(r-2)/(r-1)+(r-2)/r=2/[r(r-1)]\).
The multipliers in the two identities are reciprocal,
so \(\mathcal P_r\mathcal L_rX=X\).
This also includes \(r=2\), when both maps are the identity.

Finally, the triangle inequality and
Lemma~\ref{lem:mapbounds} imply
\[
 \begin{aligned}
 \cut{\mathcal P_rK}
 &\le\cut{K}
   +\frac{r-2}{r-1}
      \bigl(\cut{\mathcal R_xK}+\cut{\mathcal R_yK}\bigr)
   +\frac{r-2}{r}\cut{\mathcal JK}\\
 &\le
 \left(1+2\frac{r-2}{r-1}+\frac{r-2}{r}\right)\cut{K}
 \le4\cut{K}.
 \end{aligned}
 \qedhere
\]
\end{proof}

We now collect the terms involving at least two edges
into a single linear map.
For each \(J\subseteq E(H)\) with \(|J|\ge2\), choose an
edge \(\underline e_J\in J\) and fix an orientation \(e_J\).
Define
\[
 \mathcal B_UX
 =
 \frac1{r(r-1)}
 \sum_{\substack{J\subseteq E(H)\\|J|\ge2}}
 p^{m-|J|}
 \sum_{\substack{v,w\in V(H)\\v\ne w}}
 \mathcal B_{F_J;v,w;e_J}^U X,
\]
and put \(\gamma=2mp^{m-1}/[r(r-1)]\).
Expanding \(W=p+U\) in the averaged two-root density gives
the exact identity
\begin{equation}\label{eq:general-two-root}
 M_H=\gamma\mathcal L_rU+\mathcal B_UU.
\end{equation}
To verify the single-edge contribution, fix an edge of
\(H\). The ordered root placements contribute
\(2U\) when both roots are its endpoints,
\(2(r-2)\mathcal R_xU\) when only the first root is an
endpoint, and \(2(r-2)\mathcal R_yU\) when only the second
root is an endpoint. When neither root is an endpoint,
the contribution is \((r-2)(r-3)\mathcal JU\).
Multiplying by \(p^{m-1}\), summing over the \(m\) edges,
and dividing by \(r(r-1)\) yields \(\gamma\mathcal L_rU\).
The empty edge set cancels the constant \(p^m\) in \(M_H\),
while all edge sets of size at least two contribute
\(\mathcal B_UU\).

To invert the single-edge contribution, set
\(\theta=p^{-(m-1)}\) and define the correction map
\(\mathcal A_U=\gamma^{-1}\mathcal P_r\mathcal B_U\).
The kernel \(U\), and hence every unselected edge factor,
remains fixed as the argument of \(\mathcal A_U\) varies.
Lemmas~\ref{lem:mapbounds} and~\ref{lem:marginalinverse}
give
\[
 \cut{\mathcal A_UX}\le C_H\theta\cut{X}
 \qquad\text{for every real kernel }X.
\]
We next apply the edge-substitution lemma to two successive
corrections. The resulting estimate controls
\(\mathcal A_U^2U=\mathcal A_U(\mathcal A_UU)\);
it does not assert a corresponding bound for the operator
norm of \(\mathcal A_U^2\).

\begin{lemma}[Composed correction]\label{lem:squaredcorrection}
With the notation above,
\[
 \cut{\mathcal A_U^2U}
 \le C_H\theta^2
 \bigl(d^{5/4}+A_2+\sqrt{A_2d}\bigr).
\]
\end{lemma}

\begin{proof}
If \(m=1\), then \(\mathcal B_U=0\), so the conclusion
is immediate. Assume henceforth that \(m\ge2\).
Throughout the composition, the background kernel \(U\)
and all selected oriented edges remain fixed. Thus
\[
 \mathcal A_U^2U
 =
 \gamma^{-2}\mathcal P_r\mathcal B_U
 \mathcal P_r\mathcal B_UU.
\]

Write \(\mathcal B_U=\sum_s c_s\mathcal B_s\), where
each \(\mathcal B_s\) is a selected-edge map associated
with a graph \(F_J\) having at least two edges.
The coefficients are nonnegative, and the normalization
over ordered root pairs gives
\[
 \sum_s c_s
 =
 \sum_{\substack{J\subseteq E(H)\\|J|\ge2}}p^{m-|J|}
 \le 2^m-m-1=:N_H.
\]
Also write
\(\mathcal P_r=\sum_{M\in\mathscr M}\eta_M M\), where
\(\mathscr M=\{I,\mathcal R_x,\mathcal R_y,\mathcal J\}\).
Its coefficients satisfy
\[
 \sum_{M\in\mathscr M}|\eta_M|
 =
 1+2\frac{r-2}{r-1}+\frac{r-2}{r}
 \le4.
\]

Consider a term
\(\mathcal B_s\bigl(M(\mathcal B_tU)\bigr)\)
in the expansion of
\(\mathcal B_U\mathcal P_r\mathcal B_UU\).
The kernel \(\mathcal B_tU\) is a two-root density of
a simple graph with at least two edges.
By the graph representations in Lemma~\ref{lem:mapbounds},
applying \(M\) gives another two-root density, represented
by retaining all original vertices and edges and adding
up to two isolated roots.
Any old root that ceases to be distinguished remains
an integrated vertex. Consequently, the representing
graph is simple, has distinct roots, and still has at
least two edges.

The outer map \(\mathcal B_s\) is likewise associated
with a simple graph having at least two edges.
Using fresh disjoint copies of these two graphs, we may
therefore apply Lemma~\ref{lem:substitution}.
This includes the cases in which one or both inner roots
are newly added isolated vertices. Hence
\[
 \cut{\mathcal B_s\bigl(M(\mathcal B_tU)\bigr)}
 \le C_*\bigl(d^{5/4}+A_2+\sqrt{A_2d}\bigr).
\]
The triangle inequality and the coefficient bounds above
now yield
\[
 \cut{\mathcal B_U\mathcal P_r\mathcal B_UU}
 \le4C_*N_H^2
 \bigl(d^{5/4}+A_2+\sqrt{A_2d}\bigr).
\]

Finally, apply the cut-norm bound for the outer
\(\mathcal P_r\) from Lemma~\ref{lem:marginalinverse}.
Since
\(\gamma^{-2}=[r(r-1)/(2m)]^2\theta^2\), we obtain
\[
 \begin{aligned}
 \cut{\mathcal A_U^2U}
 &\le16C_*N_H^2\gamma^{-2}
       \bigl(d^{5/4}+A_2+\sqrt{A_2d}\bigr)\\
 &\le C_H\theta^2
       \bigl(d^{5/4}+A_2+\sqrt{A_2d}\bigr),
 \end{aligned}
\]
where \(C_H\) depends only on \(H\).
\end{proof}

We can now combine the degree estimate and the composed
correction bound to complete the general local argument.

\begin{proof}[Proof of Proposition~\ref{prop:general-local}]
Set \(h=\gamma^{-1}\mathcal P_rM_H\).
Applying \(\gamma^{-1}\mathcal P_r\) to
\eqref{eq:general-two-root} and using
Lemma~\ref{lem:marginalinverse}, we obtain
\(h=U+\mathcal A_UU\).
Moreover, Corollary~\ref{cor:rooted} and the cut-norm
bound for \(\mathcal P_r\) give
\(\cut{h}\le C_H\theta\xi\).

With the background kernel \(U\) fixed, the map
\(\mathcal A_U\) is linear. Applying it to the identity
for \(h\) and rearranging yields the exact cancellation
\[
 U=h-\mathcal A_Uh+\mathcal A_U^2U.
\]
This identity requires no smallness assumption on the
operator norm of \(\mathcal A_U\).
Using
\(\cut{\mathcal A_UX}\le C_H\theta\cut{X}\),
Lemma~\ref{lem:squaredcorrection}, and \(\theta\ge1\),
we deduce
\[
 \begin{aligned}
 d
 &\le\cut{h}+\cut{\mathcal A_Uh}
       +\cut{\mathcal A_U^2U}\\
 &\le C_H\theta^2\xi
       +C_H\theta^2
        \bigl(d^{5/4}+A_2+\sqrt{A_2d}\bigr).
 \end{aligned}
\]

It remains to control the terms involving \(A_2\).
By Lemma~\ref{lem:generaldegree},
\(A_2\le C_H\theta^2(\xi+d^2)\).
Consequently,
\[
 \begin{aligned}
 \theta^2A_2
 &\le C_H\theta^4(\xi+d^2),\\
 \theta^2\sqrt{A_2d}
 &\le C_H\theta^3\sqrt{\xi d+d^3}
 \le C_H\theta^3\sqrt{\xi d}
      +C_H\theta^3d^{3/2}.
 \end{aligned}
\]
Since \(0\le d\le1\), both \(d^2\) and \(d^{3/2}\)
are at most \(d^{5/4}\).
Together with \(\theta\ge1\), these estimates give
\[
 d\le C_H\theta^4(\xi+d^{5/4})
      +C_H\theta^3\sqrt{\xi d}.
\]
Young's inequality bounds the last term by
\(d/2+C_H\theta^6\xi\).
Absorbing \(d/2\) into the left-hand side and using
\(\theta^4\le\theta^6\), we conclude that
\[
 d\le C_H\theta^6(\xi+d^{5/4})
   =C_Hp^{-6(m-1)}(\xi+d^{5/4})
   \le C_Hp^{-6m}(\xi+d^{5/4}).
\]
Increasing \(C_H\) to at least one and taking
\(L(H,p)=C_Hp^{-6m}\) proves the proposition.
\end{proof}

\section{From local estimates to hereditary quasirandomness}\label{sec:reduction}

We now apply Proposition~\ref{prop:general-local} to finite
graphs. The argument first converts hereditary subset
counts into weighted homomorphism counts, then controls
the resulting finite-size error by edge discrepancy.
A qualitative quasirandomness threshold makes the
discrepancy sufficiently small to absorb the superlinear
remainder, yielding the global linear bound.

The first step is the following conversion from subset
tests to \([0,1]\)-valued vertex weights.

\begin{lemma}\label{lem:weighted}
Let \(H\) be a graph with \(r\) vertices and \(m\ge1\) edges,
and let \(p\in(0,1)\). Suppose that an \(n\)-vertex graph
\(G\) satisfies \(\Pstar_{H,p}(\delta)\), and let \(W\)
be its adjacency kernel on the uniform probability space
\(V(G)\), with \(W(x,x)=0\). Then
\[
 \Def_{H,p}(W)\le\delta+\frac{r(r-1)}{n}.
\]
The same bound holds if \(N_H(G[S])\) in
\eqref{eq:hereditary} is replaced by the number of all
edge-preserving maps \(V(H)\to S\), without requiring
injectivity.
\end{lemma}

\begin{proof}
For each pair of distinct vertices of \(H\), exactly
\(n^{r-1}\) maps \(V(H)\to V(G)\) identify that pair.
By the union bound, the number of noninjective maps is
therefore at most \(D=\binom r2 n^{r-1}\).

Fix \(f\colon V(G)\to[0,1]\), and form a random subset
\(S_f\) by including each \(x\in V(G)\) independently
with probability \(f(x)\).
Put \(s=\sum_{x\in V(G)}f(x)\), so that
\(\int f\,d\mu=s/n\).
Write \(N(S)\) for the number of edge-preserving maps
into \(S\) under the counting convention used in the
hypothesis.

We first compare \(n^rt(H,W;f)\) with
\(\mathbb E N(S_f)\).
For every injective edge-preserving map \(\varphi\),
both expressions assign the contribution
\(\prod_{v\in V(H)}f(\varphi(v))\).
Under the injective counting convention, the remaining
contributions occur only in the weighted homomorphism
density and are each at most one.
Under the homomorphism counting convention, a
noninjective edge-preserving map contributes
\[
 \prod_{v\in V(H)}f(\varphi(v))
 \quad\text{and}\quad
 \prod_{x\in\varphi(V(H))}f(x)
\]
to the two expressions, respectively.
Both products lie in \([0,1]\), so their difference
has absolute value at most one.
Thus, under either convention,
\[
 \bigl|n^rt(H,W;f)-\mathbb E N(S_f)\bigr|\le D.
\]

Similarly, expand \(\mathbb E|S_f|^r\) and \(s^r\) as
sums over ordered \(r\)-tuples of vertices.
The contributions agree when all entries are distinct.
For each tuple with a repeated entry, the two
contributions belong to \([0,1]\), so their difference
has absolute value at most one.
Consequently,
\(\bigl|\mathbb E|S_f|^r-s^r\bigr|\le D\).

The hereditary hypothesis holds for every realization
of \(S_f\). Taking expectations therefore gives
\[
 \left|\mathbb E N(S_f)-p^m\mathbb E|S_f|^r\right|
 \le
 \mathbb E\left|N(S_f)-p^m|S_f|^r\right|
 \le\delta n^r.
\]
Combining the three estimates and using \(p^m\le1\),
we obtain
\[
 \begin{aligned}
 \left|t(H,W;f)
       -p^m\left(\int f\,d\mu\right)^r\right|
 &\le\delta+\frac{(1+p^m)D}{n^r}\\
 &\le\delta+\frac{r(r-1)}{n}.
 \end{aligned}
\]
Taking the supremum over \(f\colon V(G)\to[0,1]\)
proves both versions of the lemma.
\end{proof}

\rev{We use the standard counting lemma to transfer cut-norm
control to weighted counts; see~\cite[Lemma~10.24]{LovaszBook}.
For the prescribed-set version, see~\cite[Lemma~10]{He}.
We record the two-kernel, signed-weight form and retain its
full adaptation to the present setting.}

\begin{lemma}[Weighted counting; \rev{cf.~{\cite[Lemma~10.24]{LovaszBook}}}]\label{lem:counting}
Let \(F\) be a finite simple graph with \(\ell\) edges,
and let \(W,W'\colon\Omega^2\to[0,1]\) be symmetric
measurable kernels on a probability space \((\Omega,\mu)\).
For measurable vertex weights
\(f_v\colon\Omega\to\mathbb R\) satisfying
\(\norm{f_v}{\infty}\le1\), we have
\[
 \bigl|t(F,W;\mathbf f)-t(F,W';\mathbf f)\bigr|
 \le4\ell\,\cut{W-W'}.
\]
Here the cut norm is defined by taking the supremum
over measurable subsets \(A,B\subseteq\Omega\).
\end{lemma}

\begin{proof}
\begin{revision}
The assertion is immediate if \(\ell=0\). Otherwise telescope
the difference of the products over the \(\ell\) edges, as in
the standard counting lemma. In each term one edge carries
\(W-W'\), and all other edges carry \(W\) or \(W'\).
Fix all variables except the endpoints \(x,y\) of that edge.
Simplicity of \(F\) leaves an integrand
\(c\,a(x)(W-W')(x,y)b(y)\), with
\(|c|,\norm{a}{\infty},\norm{b}{\infty}\le1\), including
signed vertex weights. The level-set form of
\eqref{eq:cutdual}, valid on any probability space, bounds
this integral by \(4\cut{W-W'}\).
Fubini's theorem and summation over the edges prove the claim.
\end{revision}
\end{proof}

The conversion in Lemma~\ref{lem:weighted} introduces
an error of order \(n^{-1}\).
The following lower bound for the discrepancy of a
finite simple graph allows us to control this error
by a quadratic term in \(\Disc_p(G)\).
\rev{The argument is a standard Rademacher second- and
fourth-moment estimate (compare~\cite[Section~3]{AN}); we
retain the calculation for the explicit finite-size constant.}

\begin{lemma}[Finite-size discrepancy]\label{lem:finitecut}
Let \(G\) be a simple graph on \(n\ge1\) vertices, and
let \(p\in(0,1)\). Put \(\alpha_p=\min\{p,1-p\}\).
Then
\[
 \Disc_p(G)\ge\frac{\alpha_p}{4\sqrt{3n}}.
\]
Equivalently,
\[
 \frac1n\le\beta_p\Disc_p(G)^2,
 \qquad
 \beta_p=\frac{48}{\alpha_p^2}.
\]
\end{lemma}

\begin{proof}
Let \(W\) be the adjacency kernel of \(G\) on the uniform
probability space \(V(G)\), and put \(U=W-p\).
Every entry of \(U\), including each diagonal entry,
belongs to \(\{-p,1-p\}\), so \(|U(x,y)|\ge\alpha_p\).

Let \((\varepsilon_y)_{y\in V(G)}\) be independent random
signs, each taking the values \(-1\) and \(1\) with equal
probability. For each \(x\in V(G)\), set
\(Z_x=\sum_yU(x,y)\varepsilon_y\) and
\(v_x=\sum_yU(x,y)^2\).
Independence and symmetry of the signs give
\[
 \mathbb E Z_x^2=v_x\ge n\alpha_p^2,
 \qquad
 \mathbb E Z_x^4
 =3v_x^2-2\sum_yU(x,y)^4
 \le3v_x^2.
\]
By H\"older's inequality with exponents \(3/2\) and \(3\),
\[
 \mathbb E Z_x^2
 =\mathbb E\bigl(|Z_x|^{2/3}|Z_x|^{4/3}\bigr)
 \le(\mathbb E|Z_x|)^{2/3}
     (\mathbb E|Z_x|^4)^{1/3}.
\]
Since \(v_x>0\), rearranging and using the moment bounds
yields
\[
 \mathbb E|Z_x|
 \ge\frac{v_x^{3/2}}{(\mathbb E Z_x^4)^{1/2}}
 \ge\sqrt{\frac{v_x}{3}}
 \ge\alpha_p\sqrt{\frac n3}.
\]

For every realization of the signs,
\(\norm{\varepsilon}{\infty}=1\) and
\(T_U\varepsilon(x)=Z_x/n\).
Thus, with respect to the uniform measure,
\[
 \norm{T_U}{\infty\to1}
 \ge\norm{T_U\varepsilon}{1}
 =\frac1{n^2}\sum_x|Z_x|.
\]
Taking expectations and applying \eqref{eq:cutdual},
together with \(\cut{U}=\Disc_p(G)\), gives
\[
 4\Disc_p(G)
 \ge\norm{T_U}{\infty\to1}
 \ge\frac1{n^2}\sum_x\mathbb E|Z_x|
 \ge\frac{\alpha_p}{\sqrt{3n}}.
\]
This proves the first inequality. Squaring and rearranging
gives the equivalent bound
\(n^{-1}\le48\alpha_p^{-2}\Disc_p(G)^2\).
\end{proof}

We use the following cut-discrepancy formulation of the
hereditary quasirandomness theorem of Simonovits and
S\'os~\cite{SS}; see also the stronger quantitative
version in~\cite[Theorem~1.2]{CFS}.
The case \(H=K_2\) is immediate.

\begin{lemma}[\rev{Simonovits--S\'os {\cite{SS}}; see also {\cite[Theorem~1.2]{CFS}}}]\label{lem:qualitative}
Let \(H\) be a fixed graph with at least one edge, and
let \(p\in(0,1)\). For every \(\varepsilon>0\), there
exist \(\delta_0=\delta_0(H,p,\varepsilon)>0\) and
an integer \(n_0=n_0(H,p,\varepsilon)\ge1\) such that
every \(n\)-vertex graph \(G\) with \(n\ge n_0\) satisfies
\[
 \Pstar_{H,p}(\delta_0)
 \quad\Longrightarrow\quad
 \Disc_p(G)\le\varepsilon.
\]
The same choices of \(\delta_0\) and \(n_0\) may be used
whether the hereditary condition counts injective
edge-preserving maps or all edge-preserving maps.
\end{lemma}

\begin{revision}
The cited theorem gives this formulation after two standard
conversions. If \(K=W-p\) and
\(q(S)=\int_{S\times S}K\,d\mu^2\), symmetry gives, even for
overlapping sets \(A,B\),
\[
 \int_{A\times B}K\,d\mu^2
 =\tfrac12\bigl(q(A\cup B)+q(A\cap B)
                 -q(A\setminus B)-q(B\setminus A)\bigr).
\]
Thus cut discrepancy is at most twice the maximum internal
discrepancy. The two counting conventions differ, after
normalization, by at most \(\binom{v(H)}2/n\). Taking half
of an admissible injective-count error as \(\delta_0\) and
increasing \(n_0\) to absorb the other half gives a common
choice for both conventions. No quantitative strengthening
of the cited qualitative theorem is used at this step.
\end{revision}

We apply the qualitative theorem only at a fixed accuracy
depending on \(H\) and \(p\).
Once the discrepancy is below this threshold, the local
estimate gives linear dependence on the hereditary error.

\begin{proof}[Proof of Theorem~\ref{thm:main}]
Write \(r=v(H)\) and \(m=e(H)\ge1\).
Let \(G\) be an \(n\)-vertex graph satisfying
\(\Pstar_{H,p}(\delta)\), under either counting convention.
Let \(W\) be its adjacency kernel on the uniform probability
space \(V(G)\), and put \(d=\cut{W-p}=\Disc_p(G)\).

By Lemma~\ref{lem:weighted} and
Proposition~\ref{prop:general-local}, there is a constant
\(L=L(H,p)\ge1\) such that
\[
 d\le L\left(\delta+\frac{r(r-1)}n+d^{5/4}\right).
\]
Lemma~\ref{lem:finitecut} gives
\(n^{-1}\le\beta_p d^2\).
Since \(0\le d\le1\), we have \(d^2\le d^{5/4}\), and
therefore
\begin{equation}\label{eq:absorb}
 d\le L\delta+Qd^{5/4},
 \qquad
 Q=L\bigl(1+r(r-1)\beta_p\bigr).
\end{equation}
Thus the finite-size error has been incorporated into
the superlinear remainder.

Set \(\varepsilon_*=(2Q)^{-4}\).
Apply Lemma~\ref{lem:qualitative} at this accuracy, and
choose \(\delta_0>0\) and an integer \(n_0\ge1\) that
work for both counting conventions.
Define
\(\delta_*=\min\{\delta_0,p^m n_0^{-r},1\}\).
All these parameters depend only on \(H\) and \(p\).

To apply the qualitative theorem when \(\delta\) is small,
we first observe that the hereditary hypothesis forces
\(\delta\ge p^m/n^r\).
Indeed, for a singleton \(S\subseteq V(G)\), the number
of edge-preserving maps \(V(H)\to S\) is zero, since
\(H\) has an edge and \(G\) has no loops.
Thus either counting convention in
\eqref{eq:hereditary} gives \(p^m\le\delta n^r\).

Suppose first that \(\delta<\delta_*\).
The preceding lower bound, together with
\(\delta<p^m n_0^{-r}\), implies \(n>n_0\).
Also, \(\delta<\delta_0\), so \(G\) satisfies
\(\Pstar_{H,p}(\delta_0)\).
Lemma~\ref{lem:qualitative} therefore gives
\(d\le\varepsilon_*\), and hence
\[
 Qd^{5/4}
 =Qd^{1/4}d
 \le Q\varepsilon_*^{1/4}d
 =\frac d2.
\]
Absorbing this term in \eqref{eq:absorb} yields
\(d\le2L\delta\).

If \(\delta\ge\delta_*\), the bound \(d\le1\) instead
gives \(d\le\delta_*^{-1}\delta\).
Combining the two cases proves
\(\Disc_p(G)\le C(H,p)\delta\), where
\(C(H,p)=\max\{2L,\delta_*^{-1}\}\).
The same constant works for both counting conventions
and for every \(n\ge1\), with no additive finite-size
error.
\end{proof}

Corollary~\ref{cor:linear} follows directly from
Theorem~\ref{thm:main}: taking \(A=B=S\) and using
\(N_{K_2}(G[S])=e_G(S,S)\) gives
\[
 \bigl|N_{K_2}(G[S])-p|S|^2\bigr|
 \le\Disc_p(G)n^2
 \le C(H,p)\delta n^2.
\]

The prescribed-set conclusion follows by combining this
discrepancy bound with weighted counting.

\begin{proof}[Proof of Corollary~\ref{cor:prescribed}]
Let \(W\) be the adjacency kernel of \(G\) on the uniform
probability space \(V(G)\), and put \(d=\Disc_p(G)\).
Fix arbitrary subsets \(S_1,\ldots,S_s\subseteq V(G)\),
and set \(f_i=\one_{S_i}\).
Write \(N_F^{\mathrm{hom}}(G;S_1,\ldots,S_s)\) for the
number of edge-preserving maps that send vertex \(i\)
into \(S_i\), without requiring injectivity.
Then
\(n^{-s}N_F^{\mathrm{hom}}(G;S_1,\ldots,S_s)
=t(F,W;\mathbf f)\).
Applying Lemma~\ref{lem:counting} with \(W'\equiv p\)
gives
\[
 \left|t(F,W;\mathbf f)
       -p^\ell\prod_{i=1}^s\frac{|S_i|}{n}\right|
 \le4\ell d.
\]

For injective counts, the difference from the homomorphism
count is at most the total number of noninjective maps
\(V(F)\to V(G)\).
A union bound over pairs of vertices gives
\[
 0\le
 t(F,W;\mathbf f)
 -\frac{N_F(G;S_1,\ldots,S_s)}{n^s}
 \le\frac{\binom{s}{2}}n.
\]
This estimate holds regardless of overlaps among the
sets \(S_i\).
By Lemma~\ref{lem:finitecut} and \(d\le1\),
\(\binom{s}{2}/n\le\binom{s}{2}\beta_p d^2
\le\binom{s}{2}\beta_p d\).
Consequently,
\[
 \left|
 \frac{N_F(G;S_1,\ldots,S_s)}{n^s}
 -p^\ell\prod_{i=1}^s\frac{|S_i|}{n}
 \right|
 \le
 \left(4\ell+\binom{s}{2}\beta_p\right)d.
\]
For the homomorphism conclusion, the same estimate holds
with the collision term omitted.

Finally, Theorem~\ref{thm:main} gives
\(d\le C(H,p)\delta\) under either counting convention
in the hypothesis.
Substituting this bound and multiplying by \(n^s\)
proves the assertion, with a constant depending only
on \(H,F,p\).
\end{proof}

\section{Stronger estimates for triangle-free graphs}\label{sec:proof}

We now sharpen Proposition~\ref{prop:general-local} for
triangle-free graphs. A direct expansion of the star
terms improves both the exponent of the remainder and
the dependence of the local constant on \(p\).
These estimates provide an additional refinement and
are not needed for the proof of Theorem~\ref{thm:main}.

\begin{proposition}\label{prop:local}
Let \(H\) be a fixed triangle-free graph with \(m\ge1\)
edges, and let \(p\in(0,1)\).
There exists a constant \(L=L(H,p)\), with
\(1\le L(H,p)\le c_Hp^{-3m}\) for some \(c_H\ge1\)
depending only on \(H\), such that every symmetric kernel
\(W\colon\Omega^2\to[0,1]\) on a finite probability space
satisfies
\[
 \cut{W-p}
 \le L\bigl(\Def_{H,p}(W)+\cut{W-p}^{3/2}\bigr).
\]
\end{proposition}

Throughout this section, let \(H\) be triangle-free,
with \(r=v(H)\) and \(m=e(H)\ge1\).
As before, write \(U=W-p\), \(T=T_U\), \(d=\cut{U}\),
\(\xi=\Def_{H,p}(W)\), and \(a=T\one\).
Set \(\mathcal E=\xi+d^{3/2}\).
For integers \(k\ge1\), define
\(A_k=\int_\Omega a^k\,d\mu\), and write
\(\mu_1=A_1\).

To organize the star terms, put
\(s_k=\sum_{v\in V(H)}\binom{\deg_H(v)}{k}\) for \(k\ge1\),
and set \(\kappa=s_2/(mp)\).
For \(k\ge2\), a \(k\)-edge star has a unique centre,
so \(s_k\) counts the edge sets of \(H\) that form such
a star. In particular, \(s_k=0\) for \(k>\Delta(H)\).

The following elementary bounds will be used repeatedly:
\begin{equation}\label{eq:abasic}
 \begin{gathered}
 -p\le a\le1-p,\qquad
 \norm{a}{1}\le4d,\qquad
 |\mu_1|\le d,\\
 |A_k|\le\norm{a}{1}\qquad(k\ge1).
 \end{gathered}
\end{equation}
The range of \(a\) follows from \(0\le W\le1\), and
Lemma~\ref{lem:operators} gives
\(\norm{a}{1}=\norm{T\one}{1}\le4d\).
The bound on \(\mu_1\) follows by testing the cut norm
on \(\Omega\times\Omega\), while the last inequality
uses \(|a|\le1\).

For \(J\subseteq E(H)\), retain the notation
\(F_J=(V(H),J)\), including all isolated vertices.
Expanding each edge factor as \(W=p+U\) gives
\[
 t_v(H,W;x)-p^m
 =
 \sum_{\varnothing\ne J\subseteq E(H)}
 p^{m-|J|}t_v(F_J,U;x).
\]
For distinct roots \(v,w\), the same expansion holds
with \(t_v(F_J,U;x)\) replaced by
\(t_{v,w}(F_J,U;x,y)\).
We shall group these expansions according to whether
\(F_J\) is a star with isolated vertices.

We begin with the degree terms that will control the remaining
star contributions in the two-root expansion.

\begin{lemma}\label{lem:degree}
There exists a constant \(C_H>0\), depending only on \(H\),
such that
\[
 A_2+|\mu_1|+\norm{a}{1}
 \le C_Hp^{-(2m-1)}\mathcal E.
\]
\end{lemma}

\begin{proof}
We first group the one-root expansion according to its
star terms. By Lemma~\ref{lem:nonstar}, the remaining
terms have total \(L^1\) norm at most \(C_Hd^{3/2}\):
each corresponding graph is triangle-free and is not
a star with isolated vertices, and all coefficients
\(p^{m-|J|}\) are at most one.

A single edge contributes \(a\) when the root is an
endpoint and \(\mu_1\) otherwise.
For a \(k\)-edge star with \(k\ge2\), a root at its
centre, a leaf, or an isolated vertex contributes
\(a^k\), \(T(a^{k-1})\), or \(A_k\), respectively.
Summing over the roots and the star edge sets gives
\begin{equation}\label{eq:staridentity}
 \begin{split}
 rR_H={}&mp^{m-1}\bigl(2a+(r-2)\mu_1\bigr)\\
 &+\sum_{k=2}^{\Delta(H)}s_kp^{m-k}
 \bigl(a^k+kT(a^{k-1})+(r-k-1)A_k\bigr)
 +\mathcal R_1,
 \end{split}
\end{equation}
where \(\norm{\mathcal R_1}{1}\le C_Hd^{3/2}\).
All sums below are understood to be empty when
\(\Delta(H)=1\).

The polynomial from Lemma~\ref{lem:generaldegree} can
be written as
\[
 \Phi_H(t)
 =2mp^{m-1}t+
   \sum_{k=2}^{\Delta(H)}s_kp^{m-k}t^k.
\]
Thus \eqref{eq:staridentity} takes the form
\(rR_H=\Phi_H(a)+B(a)+c+\mathcal R_1\), where
\[
 \begin{aligned}
 B(a)
 &=\sum_{k=2}^{\Delta(H)}
   ks_kp^{m-k}T(a^{k-1}),\\
 c
 &=m(r-2)p^{m-1}\mu_1+
   \sum_{k=2}^{\Delta(H)}
   (r-k-1)s_kp^{m-k}A_k.
 \end{aligned}
\]
Using \(\norm{Tf}{\infty}\le\norm{f}{1}\),
\(|a|\le1\), and \eqref{eq:abasic}, we obtain
\(\norm{B(a)}{\infty}+|c|\le C_Hd\).

Multiply this identity by \(a\) and integrate.
Corollary~\ref{cor:rooted} and \(\norm{a}{\infty}\le1\)
give
\(\bigl|r\int aR_H\,d\mu\bigr|\le2rb_r\xi\).
The contribution of \(B(a)+c\) is at most
\(C_Hd\norm{a}{1}\le C_Hd^2\), and that of
\(\mathcal R_1\) is at most \(C_Hd^{3/2}\).
Applying \eqref{eq:coercivity} and using
\(d^2\le d^{3/2}\), we conclude that
\begin{equation}\label{eq:secondmoment}
 A_2\le C_Hp^{-(m-1)}\mathcal E.
\end{equation}

Next, integrate \eqref{eq:staridentity} without the
factor \(a\).
Self-adjointness of \(T\) and \(T\one=a\) imply
\(\int T(a^{k-1})\,d\mu=A_k\).
The root-placement coefficients sum to \(r\) for
each star, so
\[
 \left|
 mp^{m-1}\mu_1+
 \sum_{k=2}^{\Delta(H)}s_kp^{m-k}A_k
 \right|
 \le C_H\mathcal E.
\]
Since \(|A_k|\le A_2\) for \(k\ge2\) and
\(p^{m-k}\le1\), it follows that
\begin{equation}\label{eq:meancontrol}
 |\mu_1|
 \le C_Hp^{-(m-1)}(\mathcal E+A_2)
 \le C_Hp^{-2(m-1)}\mathcal E.
\end{equation}

It remains to control \(\norm{a}{1}\).
In \eqref{eq:staridentity}, retain the terms
\(2mp^{m-1}a\) and \(2s_2p^{m-2}Ta\).
Every other term from a star with at least two edges
has \(L^1\) norm bounded by its coefficient times
\(A_2\). Indeed,
\(\norm{a^k}{1}\le A_2\) for \(k\ge2\), and
\(\norm{T(a^{k-1})}{1}\le\norm{a^{k-1}}{1}\le A_2\)
for \(k\ge3\).

After division by \(2mp^{m-1}\), the single-edge
constant has coefficient \((r-2)/2\).
Each remaining star coefficient is at most
\(C_Hp^{1-k}\le C_Hp^{-(m-1)}\), since \(k\le m\).
Recalling \(\kappa=s_2/(mp)\), we therefore obtain
\[
 \begin{aligned}
 \norm{a+\kappa Ta}{1}
 &\le C_Hp^{-(m-1)}(\mathcal E+A_2)
      +\frac{r-2}{2}|\mu_1|\\
 &\le C_Hp^{-2(m-1)}\mathcal E,
 \end{aligned}
\]
by \eqref{eq:secondmoment} and \eqref{eq:meancontrol}.

Set \(h=a+\kappa Ta\).
Since \(a=T\one\), applying \(T\) and rearranging gives
the exact identity
\(a=h-\kappa Th+\kappa^2T^3\one\).
Lemma~\ref{lem:operators} consequently yields
\[
 \norm{a}{1}
 \le(1+\kappa)\norm{h}{1}+16\kappa^2d^2.
\]
If \(m\ge2\), then \(\kappa\le C_Hp^{-1}\),
\(d^2\le\mathcal E\), and \(2m-1\ge2\). Hence
\[
 \norm{a}{1}
 \le C_Hp^{-(2m-1)}\mathcal E
      +C_Hp^{-2}\mathcal E
 \le C_Hp^{-(2m-1)}\mathcal E.
\]
If \(m=1\), then \(s_2=\kappa=0\), so the preceding
bound on \(h=a\) gives the same conclusion.

Finally, \eqref{eq:secondmoment} and
\eqref{eq:meancontrol} have no larger power of \(p^{-1}\)
than \(2m-1\). Adding the three estimates and enlarging
\(C_H\) proves the lemma.
\end{proof}

With the degree terms controlled, we turn to the averaged
two-root density and retain its two principal star terms.

\begin{lemma}\label{lem:tworoot}
There exists a constant \(C_H>0\), depending only on \(H\),
such that
\[
 \cut{U+\kappa U^2}
 \le C_Hp^{-(3m-2)}\mathcal E,
 \qquad
 \kappa=\frac{s_2}{mp}.
\]
Here \(U^2\) denotes the composition kernel.
\end{lemma}

\begin{proof}
Expand \(r(r-1)M_H\) over all nonempty edge sets and
ordered pairs of distinct roots.
Every graph \(F_J\) in this expansion is triangle-free.
By Lemma~\ref{lem:nonstar}, the total contribution from
graphs that are not stars with isolated vertices has
cut norm at most \(C_Hd^{3/2}\).
Here and below, the number of terms depends only on
\(H\), and every coefficient \(p^{m-|J|}\) is at most one.

We retain two types of star terms.
For a single edge, placing the roots at its endpoints
gives \(U\) in each of the two possible orders.
Summing over the \(m\) edges gives
\(2mp^{m-1}U\).
For a two-edge star, placing the roots at its two leaves
and integrating the centre gives \(U^2\).
There are \(s_2\) such edge sets and two orders of the
leaf roots, giving \(2s_2p^{m-2}U^2\).

We claim that every remaining rooted star density has
cut norm at most \(\norm{a}{1}\).
First consider a star with at least two edges.
Unless it is a two-edge star with both leaves as roots,
at least one leaf is not a root.
To estimate a cut integral, place indicators
\(\one_A,\one_B\) at the two roots and integrate the
variable at such a leaf.
This produces a factor \(a\) at the centre.
All remaining edge factors and indicators have absolute
value at most one, so taking absolute values and
integrating the remaining variables bounds the cut
integral by \(\int_\Omega|a|\,d\mu\).
This argument also applies when the centre is a root.

For a single edge, all placements other than the retained
endpoint pair have at least one isolated root.
If exactly one root is an endpoint, the resulting kernel
is \(a(x)\) or \(a(y)\); if both roots are isolated, it is
the constant \(\mu_1\).
Their cut norms are at most \(\norm{a}{1}\), since
\(|\mu_1|\le\norm{a}{1}\).
Thus no choice of a centre for a single edge is needed.

Collecting the retained terms and the remainder gives
\begin{equation}\label{eq:twoidentity}
 r(r-1)M_H
 =2mp^{m-1}U+2s_2p^{m-2}U^2+\mathcal R_2,
\end{equation}
where
\(\cut{\mathcal R_2}
\le C_H(d^{3/2}+\norm{a}{1})\).
When \(m=1\), we have \(s_2=\kappa=0\), so the
\(U^2\) term is absent.

Dividing \eqref{eq:twoidentity} by \(2mp^{m-1}\)
and applying Corollary~\ref{cor:rooted} yields
\[
 \cut{U+\kappa U^2}
 \le C_Hp^{-(m-1)}
       \bigl(\xi+d^{3/2}+\norm{a}{1}\bigr).
\]
By Lemma~\ref{lem:degree},
\(\norm{a}{1}\le C_Hp^{-(2m-1)}\mathcal E\).
Since \(\mathcal E=\xi+d^{3/2}\) and \(0<p<1\),
we conclude that
\[
 \cut{U+\kappa U^2}
 \le C_Hp^{-(m-1)}p^{-(2m-1)}\mathcal E
 =C_Hp^{-(3m-2)}\mathcal E.
 \qedhere
\]
\end{proof}

The two-root estimate now completes the proof by an exact
cancellation of the quadratic composition term.

\begin{proof}[Proof of Proposition~\ref{prop:local}]
Set \(V=U+\kappa U^2\).
By associativity of kernel composition,
\[
 U=V-\kappa VU+\kappa^2U^3.
\]
Lemma~\ref{lem:operators} gives
\(\cut{VU}\le4\cut{V}\), while \eqref{eq:cutdual}
and the same lemma imply
\(\cut{U^3}\le\norm{T^3}{\infty\to1}\le16d^2\).
Since \(\kappa\ge0\), the triangle inequality yields
\[
 d\le(1+4\kappa)\cut{V}+16\kappa^2d^2.
\]

By Lemma~\ref{lem:tworoot} and
\(1+4\kappa\le C_Hp^{-1}\), the first term is at most
\(C_Hp^{-(3m-1)}\mathcal E\).
If \(m\ge2\), then \(\kappa\le C_Hp^{-1}\) and
\(d^2\le d^{3/2}\le\mathcal E\), so
\[
 16\kappa^2d^2
 \le C_Hp^{-2}\mathcal E
 \le C_Hp^{-(3m-1)}\mathcal E.
\]
If \(m=1\), then \(\kappa=0\), and this term vanishes.
Consequently,
\[
 d\le C_Hp^{-(3m-1)}(\xi+d^{3/2})
   \le C_Hp^{-3m}(\xi+d^{3/2}).
\]
Choosing \(c_H\ge\max\{1,C_H\}\) and
\(L(H,p)=c_Hp^{-3m}\) proves the proposition.
\end{proof}

\begin{remark}\label{rem:girthfive}
Suppose, in addition to being triangle-free, that \(H\)
contains no four-cycle.
Every spanning subgraph \(F_J\) inherits both properties,
so the final assertion of Lemma~\ref{lem:nonstar} bounds
all nonstar contributions by \(C_Hd^2\).
The remaining error terms in the preceding proofs are
also controlled with \(\mathcal E=\xi+d^2\), with the
same dependence on \(p\).
Thus the argument gives
\[
 \cut{W-p}
 \le c_Hp^{-3m}
 \bigl(\Def_{H,p}(W)+\cut{W-p}^2\bigr).
\]
In particular, this estimate applies to every fixed
forest with at least one edge and to every fixed graph
with at least one edge and girth at least five.
\end{remark}

\section{Sharpness and parameter dependence}\label{sec:remarks}

Perturbing the constant density shows that linear
dependence on the hereditary counting error is optimal,
as already observed in~\cite{CFS}.
For the constant kernel \(W_q\equiv q\),
\(\cut{W_q-p}=|q-p|\) and
\(\Def_{H,p}(W_q)=|q^m-p^m|\).
The following finite-graph formulation identifies the
asymptotic ratio and ensures that the counting-conversion
error is negligible.

\begin{proposition}[Sharpness]\label{prop:sharpness}
Let \(H\) be a fixed graph with \(m\ge1\) edges, and let
\(p\in(0,1)\). There exist graphs \(G_j\) of strictly
increasing orders \(n_j\) and numbers \(\delta_j>0\),
with \(\delta_j\to0\), such that \(G_j\) satisfies
\(\Pstar_{H,p}(\delta_j)\), \(n_j^{-1}=o(\delta_j)\), and
\[
 \lim_{j\to\infty}
 \frac{\Disc_p(G_j)}{\delta_j}
 =\frac1{mp^{m-1}}.
\]
\end{proposition}

\begin{proof}
Write \(r=v(H)\).
Choose \(t_j\downarrow0\) with \(0<t_j<1-p\), and put
\(q_j=p+t_j\).
We first choose graphs \(G_j\), of strictly increasing
orders \(n_j\), whose adjacency kernels \(W_j\) satisfy
\(\rho_j:=\cut{W_j-q_j}\le t_j^2\) and
\(n_j^{-1}\le t_j^2\).

To justify this choice, consider \(G(n,q)\).
For fixed \(A,B\subseteq[n]\), the ordered edge count
\(e_G(A,B)\) is a sum of independent Bernoulli variables
with coefficients in \(\{0,1,2\}\), and
\[
 \mathbb E e_G(A,B)
 =q\bigl(|A||B|-|A\cap B|\bigr).
\]
Its expectation therefore differs from \(q|A||B|\)
by at most \(n\).
The sum of the squared coefficients is at most
\(4\binom n2\le2n^2\).
For \(n\ge2/\varepsilon\), \rev{Hoeffding's inequality~\cite{Hoeffding}} and
a union bound over the at most \(4^n\) ordered pairs
\(A,B\) give
\[
 \Pr\bigl(\Disc_q(G(n,q))>\varepsilon\bigr)
 \le2\cdot4^n\exp(-\varepsilon^2n^2/4).
\]
For each fixed \(\varepsilon>0\), this bound tends to
zero as \(n\to\infty\), uniformly in \(q\in(0,1)\).
Taking \(\varepsilon=t_j^2\), we may therefore choose
\(n_j>n_{j-1}\) sufficiently large and select a realization
with the required properties.

Define the smallest admissible hereditary error by
\[
 \delta_j
 =
 \frac1{n_j^r}
 \max_{S\subseteq V(G_j)}
 \bigl|N_H(G_j[S])-p^m|S|^r\bigr|.
\]
Then \(G_j\) satisfies \(\Pstar_{H,p}(\delta_j)\)
by definition. Taking \(S\) to be a singleton gives
\(\delta_j\ge p^m/n_j^r>0\).

Apply Lemma~\ref{lem:counting} with the constant kernel
\(q_j\) and all vertex weights equal to \(\one_S\).
After accounting for noninjective maps, we obtain,
uniformly over \(S\subseteq V(G_j)\),
\[
 \left|
 \frac{N_H(G_j[S])}{n_j^r}
 -q_j^m\left(\frac{|S|}{n_j}\right)^r
 \right|
 \le4m\rho_j+\frac{\binom r2}{n_j}
 =:\eta_j.
\]
Since \(q_j>p\), this gives
\(\delta_j\le q_j^m-p^m+\eta_j\).
Taking \(S=V(G_j)\) also gives
\(\delta_j\ge q_j^m-p^m-\eta_j\).
As \(\eta_j=O_H(t_j^2)\), it follows that
\[
 \delta_j
 =q_j^m-p^m+O_H(t_j^2)
 =mp^{m-1}t_j+O_{H,p}(t_j^2).
\]

The reverse triangle inequality for the cut norm,
together with \(\cut{q_j-p}=t_j\), yields
\(\bigl|\Disc_p(G_j)-t_j\bigr|\le\rho_j\le t_j^2\).
Consequently, \(\delta_j\to0\) and
\[
 \frac{\Disc_p(G_j)}{\delta_j}
 =
 \frac{t_j+O(t_j^2)}
      {mp^{m-1}t_j+O_{H,p}(t_j^2)}
 \longrightarrow\frac1{mp^{m-1}}.
\]
Finally, \(n_j^{-1}\le t_j^2\) and
\(\delta_j\asymp_{H,p}t_j\) imply
\(n_j^{-1}=o(\delta_j)\).
\end{proof}

The balanced two-part perturbation below also appears
in~\cite[Section~4]{CFS}, with the signs of the within-part
and between-part perturbations reversed.
Here we use it to show that linear dependence on the
hereditary counting error remains optimal when the
overall density is prescribed.
Throughout, overall density means \(2e(G)/n^2\).
A small adjustment of the edge set gives an error of
at most \(n^{-2}\), and exact density when \(p\) is rational.

\begin{proposition}[Sharpness at prescribed overall density]
\label{prop:fixeddensity}
Let \(H\) be a fixed graph with \(r\) vertices and \(m\ge1\)
edges, and let \(p\in(0,1)\).
There exist graphs \(G_j\) of strictly increasing orders
\(n_j\) and numbers \(\delta_j>0\), with \(\delta_j\to0\),
such that \(G_j\) satisfies \(\Pstar_{H,p}(\delta_j)\) and
\[
 \Disc_p(G_j)=\Theta_{H,p}(\delta_j),
 \qquad
 n_j^{-1}=o(\delta_j),
\]
while
\[
 \left|\frac{2e(G_j)}{n_j^2}-p\right|\le n_j^{-2}.
\]
If \(p\) is rational, the graphs can be chosen so that
\(2e(G_j)/n_j^2=p\) for every \(j\).
\end{proposition}

\begin{proof}
On the uniform probability space \(\{+1,-1\}\), put
\(f(x)=x\) and
\(W_\varepsilon(x,y)=p+\varepsilon f(x)f(y)\), where
\(0<\varepsilon<\min\{p,1-p\}\).
Since \(\int f\,d\mu=0\), we have
\(\int W_\varepsilon\,d\mu^2=p\).
Moreover,
\[
 \cut{W_\varepsilon-p}
 =\varepsilon
   \left(\max_{A\subseteq\{+1,-1\}}
   \left|\int_A f\,d\mu\right|\right)^2
 =\frac{\varepsilon}{4}.
\]
Testing the hereditary defect with the indicator of
the positive atom gives the lower bound below, while
Lemma~\ref{lem:counting} gives the upper bound:
\[
 2^{-r}mp^{m-1}\varepsilon
 \le2^{-r}\bigl((p+\varepsilon)^m-p^m\bigr)
 \le\Def_{H,p}(W_\varepsilon)
 \le m\varepsilon.
\]

Choose \(\varepsilon_j\downarrow0\) with
\(0<\varepsilon_j<\min\{p,1-p\}\).
For an even integer \(n_j\), partition the vertex set
into two sets of size \(n_j/2\), and define \(K_j\) to
equal \(p+\varepsilon_j\) within each part, including
on the diagonal, and \(p-\varepsilon_j\) between parts.
Then \(\int K_j\,d\mu^2=p\) and
\(\cut{K_j-p}=\varepsilon_j/4\).

Choose each unordered pair of distinct vertices as an
edge independently with probability given by \(K_j\).
The concentration argument in the proof of
Proposition~\ref{prop:sharpness} applies equally to these
independent, not necessarily identically distributed
edge indicators.
The discrepancy between the expected adjacency kernel
and \(K_j\) is supported on the diagonal and has
\(L^1\) norm at most \(1/n_j\).
For each fixed \(j\), we may therefore choose an
arbitrarily large even \(n_j\) and a realization
\(\widetilde G_j\) such that
\[
 \cut{\widetilde W_j-K_j}\le\varepsilon_j^2,
 \qquad
 n_j^{-1}\le\frac{\varepsilon_j^2}{2}.
\]
Here the random fluctuation and the diagonal error
are each taken to be at most \(\varepsilon_j^2/2\).
Choose the orders strictly increasing and sufficiently
large that an integer nearest to \(pn_j^2/2\) lies
between \(0\) and \(\binom{n_j}{2}\).

Let \(\widetilde e_j=e(\widetilde G_j)\), and choose an
integer \(E_j\) nearest to \(pn_j^2/2\).
Testing the preceding cut-norm bound on the whole
vertex set gives
\[
 \left|\widetilde e_j-\frac{pn_j^2}{2}\right|
 \le\frac{\varepsilon_j^2n_j^2}{2},
 \qquad
 |\widetilde e_j-E_j|
 \le\frac{\varepsilon_j^2n_j^2}{2}+\frac12.
\]
Add or delete exactly \(|\widetilde e_j-E_j|\) edges
to obtain a simple graph \(G_j\) with \(e(G_j)=E_j\).
Each modified edge changes the normalized \(L^1\)
distance between adjacency kernels by \(2/n_j^2\).
Thus, writing \(W_j\) for the adjacency kernel of \(G_j\),
we have
\[
 \begin{aligned}
 \rho_j:=\cut{W_j-K_j}
 &\le\varepsilon_j^2
      +\frac{2|\widetilde e_j-E_j|}{n_j^2}\\
 &\le2\varepsilon_j^2+n_j^{-2}
 =O(\varepsilon_j^2).
 \end{aligned}
\]
By the choice of \(E_j\),
\(\bigl|2e(G_j)/n_j^2-p\bigr|\le n_j^{-2}\).

If \(p=a/b\) is rational, make the initial choices of
\(n_j\) among sufficiently large multiples of \(2b\).
Then \(pn_j^2/2\) is an integer, so taking
\(E_j=pn_j^2/2\) gives overall density exactly \(p\).

Define
\[
 \delta_j
 =
 \frac1{n_j^r}
 \max_{S\subseteq V(G_j)}
 \bigl|N_H(G_j[S])-p^m|S|^r\bigr|.
\]
By definition, \(G_j\) satisfies
\(\Pstar_{H,p}(\delta_j)\), and a singleton test gives
\(\delta_j\ge p^m/n_j^r>0\).
Set
\(\eta_j=4m\rho_j+\binom r2/n_j
=O_H(\varepsilon_j^2)\).
Lemma~\ref{lem:counting} and the bound on noninjective
maps give, uniformly over \(S\subseteq V(G_j)\),
\[
 \left|
 \frac{N_H(G_j[S])}{n_j^r}
 -t(H,K_j;\one_S)
 \right|
 \le\eta_j.
\]

Applying Lemma~\ref{lem:counting} once more, now to
\(K_j\) and the constant kernel \(p\), yields
\(\delta_j\le m\varepsilon_j+\eta_j\).
On the other hand, taking \(S\) to be one entire part
of the partition gives
\[
 \delta_j
 \ge2^{-r}\bigl((p+\varepsilon_j)^m-p^m\bigr)-\eta_j.
\]
Since
\((p+\varepsilon_j)^m-p^m
=mp^{m-1}\varepsilon_j+O_{H,p}(\varepsilon_j^2)\),
these bounds imply
\(\delta_j=\Theta_{H,p}(\varepsilon_j)\), and hence
\(\delta_j\to0\).

Finally, the reverse triangle inequality gives
\[
 \left|\Disc_p(G_j)-\frac{\varepsilon_j}{4}\right|
 \le\rho_j=O(\varepsilon_j^2).
\]
Thus \(\Disc_p(G_j)=\Theta_{H,p}(\delta_j)\).
Since \(n_j^{-1}=O(\varepsilon_j^2)\), we also have
\(n_j^{-1}=o(\delta_j)\), completing the proof.
\end{proof}

Proposition~\ref{prop:sharpness} shows that any constant
\(C(H,p)\) for which Theorem~\ref{thm:main} holds must
satisfy \(C(H,p)\ge1/(mp^{m-1})\).
In particular, \(C(C_4,p)\ge1/(4p^3)\).
Proposition~\ref{prop:fixeddensity} further shows that
the conclusion cannot be improved to \(o(\delta)\)
even when the overall density is prescribed up to an
error of \(n^{-2}\), or exactly when \(p\) is rational.
Together, these results establish the optimal order
of dependence on \(\delta\).
They do not establish the optimality of the exponents
in the local remainder terms or of the dependence
of the constants on \(p\).

The local estimates are uniform over the cardinality
of the underlying probability space and the masses
of its atoms.
For the global theorem, the proof gives the admissible
choice \(C(H,p)=\max\{2L,\delta_*^{-1}\}\), where
\(\delta_*\) incorporates the qualitative quasirandomness
threshold.
Thus the polynomial bounds in \(p^{-1}\) for the local
constants do not yield a corresponding polynomial
bound for the global constant.
Throughout, \(H\) and \(p\in(0,1)\) are fixed; no
uniform assertion is made in a regime where \(p\)
varies with \(n\).

\section*{Declaration of generative AI and AI-assisted technologies}
The authors used a large language model during the preparation of this manuscript
for assistance with mathematical exploration, checking arguments,
organization and language editing. The authors are responsible for
the mathematical statements, proofs, references and final text.

\clearpage
\normalsize
\appendix
\section{Cycle expansions}\label{sec:fourcycle}

This appendix makes the preceding star expansions explicit for
the four-cycle and the five-cycle.

For \(H=C_4\), the parameters in Section~\ref{sec:proof}
are \(r=m=4\), \(s_2=4\), and \(\kappa=p^{-1}\).
Define
\(b(x)=\int_\Omega[U^2(x,z)]^2\,d\mu(z)\).
This is the one-root four-cycle density in \(U\).
In particular, \(b\ge0\) and
\(\int_\Omega b\,d\mu=\tr(T^4)\).

The averaged one-root expansion is
\begin{equation}\label{eq:c4root}
 R_{C_4}
 =2p^3(a+\mu_1)
  +p^2(a^2+2Ta+A_2)+\mathcal E_1,
\end{equation}
where
\[
 \mathcal E_1
 =2p^2\mu_1a+2p(T^2a+a\,Ta)+b.
\]
The four single-edge sets give the first term in
\eqref{eq:c4root}, while the four adjacent edge pairs
give the second.
The two opposite edge pairs contribute \(2p^2\mu_1a\).
Among the four three-edge paths, a fixed vertex is an
endpoint twice and an internal vertex twice, giving
\(2p(T^2a+a\,Ta)\).
The full cycle contributes \(b\).
Thus all fifteen nonempty edge sets are accounted for.

To estimate the remainder, note that
\(\norm{\mu_1a}{1}\le4d^2\) by \eqref{eq:abasic}.
Also, \(T^2a=T^3\one\), so
Lemma~\ref{lem:operators} gives
\(\norm{T^2a}{1}\le16d^2\).
Finally,
\(\norm{a\,Ta}{1}
\le\norm{a}{1}\norm{Ta}{\infty}
\le\norm{a}{1}^2\le16d^2\).
Together with Lemma~\ref{lem:fourthtrace}, these bounds
yield
\[
 \norm{\mathcal E_1}{1}
 \le(8p^2+64p)d^2+8K_Gd^{3/2}.
\]

The corresponding averaged two-root expansion is
\[
 M_{C_4}
 =\frac23p^3U+\frac23p^2U^2+\mathcal E_2,
 \qquad
 \cut{\mathcal E_2}
 \le C\bigl(d^{3/2}+\norm{a}{1}\bigr).
\]
There are eight ordered endpoint-root placements across
the four single-edge sets, and eight ordered leaf-root
placements across the four two-edge stars.
Dividing by \(r(r-1)=12\) gives the two coefficients
above.
The remaining terms are controlled as in the proof
of Lemma~\ref{lem:tworoot}.

Combining the degree estimate with the cancellation
argument in the proof of Proposition~\ref{prop:local}
gives \(d\le C_p(\xi+d^{3/2})\).
The fixed-threshold argument of
Section~\ref{sec:reduction} then yields
\(\Disc_p(G)\le C_p\delta\).

The five-cycle illustrates both the nonbipartite case
and the quadratic remainder in
Remark~\ref{rem:girthfive}.
Here \(r=m=5\), \(s_2=5\), and again \(\kappa=p^{-1}\).
The star expansions give
\[
 \begin{aligned}
 R_{C_5}
 &=p^4(2a+3\mu_1)
   +p^3(a^2+2Ta+2A_2)+\mathcal F_1,\\
 M_{C_5}
 &=\frac12p^4U+\frac12p^3U^2+\mathcal F_2,
 \end{aligned}
\]
where
\[
 \norm{\mathcal F_1}{1}\le Cd^2,
 \qquad
 \cut{\mathcal F_2}
 \le C\bigl(d^2+\norm{a}{1}\bigr).
\]
Indeed, every nonempty spanning subgraph of \(C_5\)
that is not a star with isolated vertices contains an
induced \(P_4\) or \(2K_2\).
Lemma~\ref{lem:nonstar} therefore gives a quadratic
bound for all nonstar contributions.
The remaining two-root star terms are bounded by
\(\norm{a}{1}\), as before.

The degree polynomial is
\(\Phi_{C_5}(t)=5p^3t(2p+t)\).
For \(t\ge-p\), it satisfies
\(t\Phi_{C_5}(t)\ge5p^4t^2\).
Thus the degree and two-root arguments apply with
\(\mathcal E=\xi+d^2\), giving
\(d\le C_p(\xi+d^2)\).
Applying the same fixed-threshold argument yields
\(\Disc_p(G)\le C_p\delta\).

Here the constants \(C\) in the remainder estimates
are absolute, while \(C_p\) may depend on the fixed
density \(p\).


\clearpage
\begin{thebibliography}{99}

\bibitem{AN}
N.~Alon and A.~Naor,
Approximating the cut-norm via Grothendieck's inequality,
\emph{SIAM J. Comput.} \textbf{35} (2006), 787--803.

\bibitem{CGW}
F.~R.~K.~Chung, R.~L.~Graham and R.~M.~Wilson,
Quasi-random graphs,
\emph{Combinatorica} \textbf{9} (1989), 345--362.

\bibitem{CFS}
D.~Conlon, J.~Fox and B.~Sudakov,
Hereditary quasirandomness without regularity,
\emph{Math. Proc. Cambridge Philos. Soc.}
\textbf{164} (2018), 385--399.

\bibitem{He}
X.~He,
Linear dependence between hereditary quasirandomness conditions,
\emph{Electron. J. Combin.} \textbf{25} (2018), no.~4,
Paper~P4.12.

\bibitem{Lovasz}
L.~Lov\'asz,
Subgraph densities in signed graphons and the local
Simonovits--Sidorenko conjecture,
\emph{Electron. J. Combin.} \textbf{18} (2011), no.~1,
Paper~P127.

\bibitem{RS}
C.~Reiher and M.~Schacht,
Forcing quasirandomness with triangles,
\emph{Forum Math. Sigma} \textbf{7} (2019), Paper~e9.

\bibitem{Simon}
B.~Simon,
\emph{Trace Ideals and Their Applications},
2nd ed., Mathematical Surveys and Monographs, vol.~120,
American Mathematical Society, Providence, RI, 2005.

\bibitem{SS}
M.~Simonovits and V.~T.~S\'os,
Hereditarily extended properties, quasi-random graphs and not
necessarily induced subgraphs,
\emph{Combinatorica} \textbf{17} (1997), 577--596.


\bibitem{Andres}
{S.~D.~Andres,
On characterizing game-perfect graphs by forbidden induced subgraphs,
\emph{Contrib. Discrete Math.} \textbf{7} (2012), no.~1, 21--34.}

\bibitem{Hoeffding}
{W.~Hoeffding,
Probability inequalities for sums of bounded random variables,
\emph{J. Amer. Statist. Assoc.} \textbf{58} (1963), 13--30.}

\bibitem{LL}
{H.~Li and X.~Liu,
Forcing quasirandomness via rooted \(F\)-densities,
preprint, arXiv:2608.08679v1 (2026).}

\bibitem{Thomas}
{E.~G.~F.~Thomas,
A polarization identity for multilinear maps,
\emph{Indag. Math.} \textbf{25} (2014), 468--474.}

\bibitem{WWY}
{G.~Wang, L.~Wang and J.~Yue,
Quasirandomness from balanced and repeated subgraph counts,
preprint, arXiv:2610.28124v3 (2026).}

\bibitem{LovaszBook}
{L.~Lov\'asz,
\emph{Large Networks and Graph Limits},
American Mathematical Society Colloquium Publications,
vol.~60, American Mathematical Society, Providence, RI, 2012.}

\end{thebibliography}
\end{document}